\documentclass[a4paper,11pt,twoside,DIV=calc]{scrartcl}	

\input xy
\xyoption{all}
\usepackage{lscape}
\usepackage[latin1]{inputenc}
\usepackage[cmtip,arrow]{xy}
\usepackage{amsmath}
\usepackage{amssymb,amscd,amsthm,mathrsfs,yfonts,manfnt,pb-diagram,pb-xy}
\usepackage[nice]{nicefrac}		
\usepackage[T1]{fontenc}
\usepackage{lmodern} 					
\usepackage{tocenter}					
\ToCenter[h]{\textwidth}{\textheight}	
\usepackage{ulem}							
\usepackage{color}

\renewcommand{\emph}{\textit}		

\usepackage{ifthen}
\newboolean{comm}
\setboolean{comm}{false}
\newboolean{sol}
\setboolean{sol}{false}
\newboolean{notes}
\setboolean{notes}{false}			
\CompileMatrices  						

\newcommand{\com}{\ifthenelse{\boolean{comm}}}
\newcommand{\sol}{\ifthenelse{\boolean{sol}}}
\newcommand{\note}{\ifthenelse{\boolean{notes}}}

\newtheorem{Def}{Definition}
\newtheorem{Prop}[Def]{Proposition}

\newtheorem{Th}[Def]{Theorem}

\newtheorem{Lem}[Def]{Lemma}
\newtheorem{Cor}[Def]{Corollary}

\theoremstyle{definition}
\newtheorem{Rem}[Def]{Remark}
\newtheorem{Ex}[Def]{Example}

\newcommand{\mR}{\ensuremath{\mathbb{R}}}								
\newcommand{\mC}{\ensuremath{\mathbb{C}}}					
\newcommand{\mN}{\ensuremath{\mathbb{N}}}
\newcommand{\mZ}{\ensuremath{\mathbb{Z}}}
\newcommand{\mK}{\ensuremath{\mathbb{K}}}

\newcommand{\mc}{\mathcal}								

\DeclareMathOperator{\Image}{Im}		

\renewcommand{\Im}{\Image}							

\newcommand{\cat}{\mathcal}

\newcommand{\ee}{{\mbox{$\varepsilon$}}}

\renewcommand{\phi}{\varphi}

\newcommand{\coloneqq}{:=}

\newcommand{\fn}{\mathcal{C}}	   				  

\newcommand{\mb}{\begin{pmatrix}}					
\newcommand{\me}{\end{pmatrix}}						
\newcommand{\hsp}{\mc{H}}									
\newcommand{\bo}{\mathbb{B}(\mc{H})}			
\newcommand{\co}{\mathbb{K}(\mc{H})}			
\newcommand{\Lp}{\mathcal{L}^p}						

\newcommand{\Dom}{\mathcal{D}}						

\DeclareMathOperator{\esssup}{esssup}     

\newcommand{\sco}{\ensuremath{\mathcal{K}}}						
\newcommand{\scy}{\ensuremath{\mathcal{Z}}}						
\newcommand{\ssu}{\ensuremath{\mathcal{S}}}						
\newcommand{\sCo}{\ensuremath{\mathcal{C}}}						
\newcommand{\potimes}{\otimes_{\pi}}									
\newcommand{\inotimes}{{\otimes_{\varepsilon}}}				
\newcommand{\LC}{\mc{LC}}															

\newcommand{\mcA}{\mc{A}}															
\newcommand{\mcB}{\mc{B}}
\newcommand{\mcC}{\mc{C}}

\newcommand{\mcE}{\mc{E}}

\newcommand{\mcS}{\mc{S}}

\DeclareMathOperator{\Hom}{Hom}						
\DeclareMathOperator{\id}{id}							

\newcommand{\rrarrow}{\rightrightarrows}

\usepackage{tocenter}				
\ToCenter[h]{\textwidth}{\textheight}	

\setkomafont{section}{\bfseries\Large}		

\usepackage{scrpage2}      	
\ofoot{}					
\ifoot{}
\rehead{\pagemark}			
\cehead{MARTIN GRENSING}	
\rohead{\pagemark}			
\cohead{Universal cycles and homological invariants}

\newcommand{\cliffs}{\mathscr{S}}
\newcommand{\fun}{h}

\begin{document}

\title{\Large{\MakeUppercase{Universal cycles and homological invariants of locally convex algebras}}}
\author{\small{MARTIN GRENSING}}
\date{\small{\today}}
\maketitle
\thispagestyle{empty}
\begin{abstract} \small{Using an appropriate notion of locally convex Kasparov modules, we show how to induce isomorphisms under a large class of functors on the category of locally convex algebras; examples are obtained from spectral triples. Our considerations are based on the action of  algebraic $K$-theory on these functors, and involve compatibility properties of the induction process with this action, and with Kasparov-type products. This is based on an appropriate interpretation of the Connes-Skandalis connection formalism. As an application, we prove Bott periodicity and a Thom isomorphism for algebras of Schwartz functions.
 As a special case, this applies to the theories $kk$ for locally convex algebras considered by Cuntz.}
\end{abstract}

\section{Introduction}
The subject of this article are homology-type invariants of locally convex algebras. It analyses to what extend  locally convex bimodules are universal for such homology theories.

Let us first recall the situation in the setting of (separable) $C^*$-algebras. On this category, we have at our disposal Kasparov's bifunctor $KK$. When $KK$ is restricted to the first or second variable, one obtains $K$-homology or $K$-theory, respectively. The elements of $KK$ may be described in essentially three ways: The extension picture, the module picture and the quasihomomorphism picture. The extension picture yields, when specialised to the first variable, the $Ext$-theory of Brown, Douglas and Fillmore (\cite{MR0458196}), which is based on earlier work of Busby (\cite{MR0225175}); here classes in $Ext(A,\mC)$ are described as extensions of $A$ by the compact operators. In the module picture, classes in $KK(A,B)$ are represented by bounded operators acting on a Hilbert $B$-module, which carries a left $A$-action (\cite{KaspOp}); in the so-called Baaj-Julg picture, the operator is replaced by a regular operator on the Hilbert module. In the quasihomomorphism picture, all information is contained in a pair of homomorphisms, both from $A$ into an algebra containing $\mK\otimes B$ as an ideal, with difference in $\mK\otimes B$  (\cite{MR899916}/\cite{MR733641}). 

In each of these pictures, there is a corresponding description of the Kasparov product $KK(A,B)\times KK(B,C)\to KK(A,C)$. In the extension picture, it corresponds to a splice, or Yoneda product, of extensions (\cite{MR1001468}). In the module picture it can be defined rather explicitly using certain operators $M$ and $N$, whose existence is guaranteed by Kasparov's technical theorem. The proof of this theorem is based on properties of the category of $C^*$-algebras; so is the proof of the fundamental theorem in (\cite{MR899916}), which is used in the definition of the product in the quasihomomorphism picture. The most calculable description of the product is the one in the module picture, where it is a generalisation of the Atiyah-Singer sharp product. 

It is well known that Kasparov's bivariant $K$-functor is universal among certain functors. If $H$ is a split exact, homotopy invariant and stable functor, then every class in  $KK(A,B)$ induces a homomorphism $H(A)\to H(B)$. Furthermore, it can be shown that this induction process is compatible with Kasparov products: if we are given classes $x\in KK(A,B)$ and $y\in KK(B,C)$, then the homomorphism induced by the Kasparov product of $x$ and $y$ is the composition of the homomorphisms induced separately by $x$ and $y$.

In the locally convex setting, the situation is different. We still have at our disposal a bivariant $K$-theory, namely the theory $kk$ constructed by Cuntz  (\cite{MR1456322}/\cite{MR2240217}). The definition of $kk$ is based on linearly split extensions of locally convex algebras of arbitrary length, and the composition corresponds again to the Yoneda product of extensions. However, elements of $kk$ still appear naturally in form of what we will call locally convex Kasparov modules. For example, a spectral triple on a compact $spin^c$-manifold $M$ satisfies additional summability conditions, and therefore defines a
 locally convex $(A,B)$-Kasparov module, where $A:=\fn^\infty(M)$ and $I$ is some ideal in the algebra bounded operators on a Hilbert space. To this object, one may associate a quasihomomorphism $\phi$ of locally convex algebras, which in turn induces a homomorphism $H(\phi):H(A)\to H(B)$ for every split exact functor. Already at this stage, there are some unresolved issues. If we take $H:=kk(\,\cdot\,,B)$, then it is not clear in how far elements of the form $kk( \phi,B)$, where $\phi$ is the quasihomomorphism associated to a  locally convex $(A,B)$-Kasparov module $x$, exhaust the group $kk(A,B)$. In  \cite{CuntzMeyer}, it suggested that this is  not the case. However, as we noted above, elements of $kk$ can be represented as such locally convex bimodules, and we show that they suffice to obtain the isomorphisms under  split exact, diffotopy invariant $\Lp$-stable functors we are interested in.

We therefore develop a complete framework of locally convex Kasparov modules, together with equivalence relations between them, and study how they induce morphisms under split exact functors on the category of locally convex algebras. We show that smooth versions of Kasparov's Dirac and dual-Dirac elements appear naturally as locally convex Kasparov modules.

The crucial, and most difficult question now is to what extent the above compatibility of the induction process with Kasparov products still holds in this more general context. Essentially, the fact that in the $C^*$-setting any two composable classes in $KK$ have a product which is again represented by a Kasparov product, corresponds to the fact that the splice of two extensions of length, say two, may again be represented by a length two extension; the analogous property for a functor on locally convex algebras is very likely false.

It is thus necessary to find sufficient conditions for a locally convex Kasparov module to represent the composition $H(\psi)\circ H(\phi)$ of two  homomorphisms induced under a functor $H$ by quasihomomorphisms coming from locally convex Kasparov modules. There are two closely related approaches to this problem. The first consists in finding operators $M$ and $N$, as in Kasparov's original description of the product, that satisfy the correct algebraic relations with respect to locally convex algebras involved in the construction of the product. This approach is feasible at least to some extent, but it turns out to be very technical. 

The second approach, which we will follow here, is based on a new interpretation of the Connes-Skandalis connection formalism (\cite{MR775126}). Already in the $C^*$-setting, it yields a genuine description of the Kasparov product, based only on Kasparov's version of the theorem of Voiculescu (\cite{MR587371}). The essential idea is that the existence of a product for two (locally convex) Kasparov modules may be interpreted as the existence of an extension of a Kasparov module. Just like the connection formalism by Connes and Skandalis,  this extension condition only yields a sufficient criterion for a cycle to represent a product; the existence of such an extension remains a consequence of Kasparov's technical theorem, whose validity is restricted to the $C^*$-setting. However, the existence of the aforementioned extension is easier to show than the existence of the operators $M$ and $N$. In the case where the first cycle is represented by a class in $K$-theory, we show that it is always possible to construct such an extension, and thus show the existence of a product in complete generality. This fact yields Bott periodicity. Using the formalism of locally convex Kasparov modules, we are also able to prove a smooth version of the Thom isomorphism theorem for the  algebra of fibrewise Schwartz functions on a smooth vector bundle over a compact smooth manifold for every split exact, $\Lp$-stable, diffotopy invariant functor.

The formalism of locally convex Kasparov modules is further useful in order to define explicitly classes in $kk$, for example those related to pseudodifferential operators. This is quite difficult in general due to the very abstract definition of $kk$. Our considerations also show that $kk$ is very likely not an analogue of Kasparov's $KK$-theory, but rather of the $E$-theory of Connes and Higson \cite{MR1065438}.

I would like to thank G. Skandalis and J. Cuntz for helpful discussions and remarks.
\section{Preliminaries}\label{prelimi}
In the sequel, we will work on the category $\LC$\index{$\LC$} of locally convex algebras. By a locally convex algebra we will mean a complete locally convex vector space that is at the same time a topological algebra, \textsl{i.e.}, the multiplication is continuous  (and associative). This means that for every continuous seminorm $p$ on $\mcA$ there is a continuous seminorm $q$ on $\mcA$ such that for all $a,b\in \mcA$
$$p(ab)\leq q(a)q(b).$$

Recall that for two locally convex topological vector spaces $V$ and $W$, their projective tensor product is defined as the algebraic tensor product $V\odot W$ of $V$ and $W$ with the locally convex structure induced from the family of seminorms defined for all $x\in V\odot W$ by
$$p\odot q (x):=\inf\{\sum\limits_{i=1}^n p(v_i)q(w_i)\big| x=\sum\limits_{i=1}^n v_i\odot w_i\},$$
where $p$ and $q$ run through systems of seminorms defining the topologies of $V$ and $W$ (see \cite{MR0075539}). The completed projective tensor product of complete locally convex spaces as above will be denoted $V\potimes W$\index{$V\potimes W$}, the extension of $p\odot q$ to $V\potimes W$ will be denoted by $p\otimes q$. The topology on $V\potimes W$ is equal to the completion of the final topology for the canonical map $V\times W\to V\odot W$.

The injective tensor product of locally convex spaces is defined in \cite{Treves}, or again in the original \cite{MR0075539} \ifthenelse{\boolean{notes}}{\marginpar{define}}{}. For complete locally convex spaces $V$ and $W$, the completion of $V\odot W$ in this topology will be denoted $V\inotimes W$\index{$V\inotimes W$}; if $V$ and $W$ are Banach spaces, $V\inotimes W$ is a Banach space with a certain norm $||\cdot||_{\footnotesize{\ee}}$. Recall that for complete subspaces $V'\leq V$ and $W'\leq W$, $\inotimes$, in contrary to $\potimes$, does have the property that $V'\inotimes W'\hookrightarrow V\inotimes W$ (see \cite{Treves} Proposition 43.7).

For $C^*$-algebras $A$ and $B$, the minimal tensor product $A\otimes_* B$\index{$A\otimes_* B$} is obtained by taking the tensor product of the Hilbert spaces from a faithful representation of $A$ and $B$ and pulling back (and completing) the induced norms on $A\odot B$ (see \cite{MR1074574} for details, e.g., the independence of the presentation, nuclearity, exactness and so on); we denote the obtained $C^*$-norm by $||\cdot||_*$. The maximal tensor product $A\otimes_{\max} B$\index{$A\otimes_{\max} B$} is obtained as the completion in the norm $||\cdot||_{\max}$ given by the supremum over all norms coming from  $*$-representations of $A$ and $B$. If $B$ is nuclear we will often not distinguish between these two $C^*$-tensor products.

\note{
Recall that for $x\in A\odot B$ we have inequalities
$$||x||_\pi\geq ||x||_{\max}\geq ||x||_*\geq||x||_{\footnotesize{\ee}},$$
Also, the maximal $C^*$-tensor product is a quotient of the minimal tensor product (the image of a *-homomorphisms of $C^*$-algebras being closed). We will have to view, for example, the projective tensor product $\mcA\potimes\Lp$, where $\mcA$ is a subalgebra of a $C^*$-algebra $A$ and $\Lp$ denotes the p-Schatten class, as a subspaces of the  $C^*$-tensor product $A\otimes \mK$. For this, we need a notion ensuring injectivity of certain natural maps. As the space $\Lp$ is an infinite dimensional Banach space, it is non-nuclear in the sense of Grothendieck (\cite{MR0075539}: a locally convex space $V$ is called nuclear in the sense of Grothendieck, if as functors on the category of locally convex spaces we have $\potimes V= \inotimes V$). Hence we need a weaker property called the approximation property (see (\cite{MR1874893})): A Banach space  has the  approximation property  if the compact operators on it are the closure of the finite rank operators. Note that for symmetrically normed operator ideals the approximation property can be easily checked by using the classification via gauge-functions exposed in \cite{Simontrace}. The following fact is well known:
\begin{Prop} Let $V$ have the  approximation property, then for every locally convex vector space $W$ the canonical map $V\potimes W\to V\inotimes W$ is injective .
\end{Prop}
\begin{Prop}\label{embedding}Let $\mcA$ be a complete locally convex subalgebra of a $C^*$-algebra $A$, $B$ a $C^*$-algebra containing $\mathcal{L}^p$. Then  the canonical map $\mcA\potimes \Lp\to A\otimes_* B$ is injective.
\end{Prop}
\begin{proof}
We have a commutative diagram
\[\xymatrix{\mcA\potimes\mc{L}^p\ar[r]^\sigma\ar[d]_{\mu}& A\otimes_* B\ar[d]\\
						\mcA\inotimes \mc{L}^p\ar[r]_\nu& A\inotimes B}\]
where $\sigma$ is the composition 
$$\mcA\potimes \Lp\to A\potimes B \to A\otimes_* B.$$ 
Now $\mu$ is injective because $\mc{L}^p$ has the approximation property, and $\nu$ is injective because of the properties of the injective tensor product. Hence $\sigma$ is injective.
\end{proof}
The above proposition will be used (mostly) without further mention in the sequel. }{}
We denote by $\scy$ the locally convex algebra of differentiable functions from $[0,1]$ all of whose derivatives vanish at the endpoints, by $\sCo$ those that are zero in $1$, and by $\ssu$ those that are zero in both endpoints; ; we define $\ssu^n$ as the $n$-fold tensor product of $\ssu^n$ with itself, similarly for $\sCo^n$ and $\scy^n$. We further get functors mapping a locally convex algebra $\mcA$ to \index{$\scy \mcA$}$\scy \mcA$, \index{$\sCo \mcA$}$\sCo \mcA$ and \index{$\ssu \mcA$}$\ssu \mcA$ as in the $C^*$-setting. We write $ev^{\mcA}_t:\scy \mcA\to \mcA$\index{$ev^{\mcA}_t$} for the evaluations in $t\in I$. 

\begin{Def} Let $\phi_0,\phi_1:\mcA\to \mcB$ be homomorphisms of locally convex algebras. A diffotopy between $\phi_0$ and $\phi_1$ is by definition a homomorphism $\Phi:\mcA\to \scy \mcB$ such that $ev^{\mcB}_i\circ\Phi=\phi_i$ for $i=0,1$.
\end{Def}
We note from \cite{MR2387048} that this is not the same as a family of morphisms that is pointwise differentiable. 

\begin{Rem}\label{decompose} The well known decomposition for Hilbert spaces into matrices with respect to a projection carries over to locally convex algebras as follows:

Let $a\in \mcA$ and $p$ an idempotent in some larger algebra containing $\mcA$. Set
$$ a_1\coloneqq pap,\; a_2\coloneqq pa(1-p),\; a_3\coloneqq (1-p)ap,\; a_4\coloneqq (1-p)a(1-p),$$ which entails $a_1+a_2+a_3+a_4=a.$
As all of the subalgebras $p\mcA p$, $(1-p)\mcA(1-p)$, $(1-p)\mcA p$, and $p\mcA(1-p)$ have empty intersection, $\mcA$ is isomorphic to the inner direct sum of them (as a locally convex space). If we write the $a$ as a $2\times 2$-matrix
$$a=\mb a_1&a_2\\a_3&a_4\me, $$
then the multiplication on $\mcA$ corresponds to the multiplication of matrices. This is usually called a Morita context (see e.g. \cite{MR2052770}); we have reserved this expression for a more analytic version of a context below.
\end{Rem}

\section{Functors and exact sequences}
We collect some general statements concerning functors and exact sequences of locally convex algebras.
\begin{Def}\label{semisplit} An exact sequence 
\[\xymatrix{0\ar[r]&{\mcC}\ar[r]&{\mcB}\ar[r]&\mcA\ar[r]&0}\]
of locally convex algebras will be called semisplit if it is split as a sequence of topological vector spaces. It will be called a split exact sequence, if it is split in the category of locally convex algebras. It will be called a double split sequence, if there are two splits for the quotient map.
\end{Def}

\begin{Def} A functor $H$ on the category of locally convex algebras with values in the category of abelian groups will be called split exact, if  for every split exact sequence
\[\xymatrix{ 0\ar[r]&{\mcC}\ar[r]&{\mcB}\ar[r]&\mcA\ar[r]&0}\]
of locally convex algebras the sequence
\[\xymatrix{ 0\ar[r]&H({\mcC})\ar[r]&H({\mcB})\ar[r]&H(\mcA)\ar[r]&0}\]
is exact.
\end{Def}
\begin{Rem}\label{Stuff}
If $H$ is split exact, then the sequence obtained by applying $H$  automatically splits.

\end{Rem}
\begin{Def}\label{functorprops} Let $H$ be a functor from locally convex algebras to abelian groups, and $J$ a locally convex algebra containing an idempotent $p$. Then $H$ is called
\begin{itemize}
\item $J$-stable if $H(\iota)$ is an isomorphism, denoted $\theta^J_\mcA$, for every locally convex algebra $\mcA$, where $\iota:{\mcA}\to {\mcA}\otimes_\pi J,\;a\mapsto a\otimes p$.
\item invariant under inner automorphisms if $H(Ad_U)=\id_{H({\mcA})}$ for every invertible $U\in {\mcA}$,
\item  pointwise diffotopy invariant, if for every family $\phi_t:{\mcA}\to {\mcB}$ of morphisms of locally convex algebras such that $\phi_t(a)$ is smooth for all $a\in {\mcA}$, $H(\phi_0)=H(\phi_1)$.
\item  diffotopy invariant if for every homomorphism $\phi:{\mcA}\to \scy {\mcB}$ we have $H(ev_0)\circ H(\phi)=H(ev_1)\circ H(\phi)$.
\end{itemize}
\end{Def}
\begin{Rem} 

If $P$ is a class of idempotents in  $J$ that are all conjugate to another, in the sense that $J$ is an ideal in a larger algebra $\hat J$ such that for any two minimal projections $p$ and $q$ there is an invertible element $u\in \hat J$ with $u^{-1}pu=q$, then the definition does not depend on the choice of idempotent in $P$ if we suppose that $H$ is invariant under conjugation by such $u$ (of course it suffices that conjugation by $u$ is continuous ${\mcB}\to {\mcB}$, so one can weaken the ideal condition).

For example, the minimal projections in the compacts or the $p$-summable operators have this property.
\end{Rem}

We then have the following properties:
\begin{Prop}\label{functorgeneralities} On the category of locally convex algebras, if $H$ is a functor with values in abelian groups, then
\begin{enumerate}
\item\label{sum} if $H$ is split exact and  $\phi$ and $\psi$ are orthogonal homomorphisms,  $H(\phi+\psi)=H(\phi)+H(\psi)$ 
\item if $H$ is $M_n$ stable for some $n$, then $H$ is $M_m$-stable for all $m\in\mN$
\item\label{Mtwo} if $H$ is $M_2$-stable, then it is invariant under inner automorphisms
\item\label{naturalinclusion} if $H$ is $\mathcal{L}^p$-stable , then the  natural map $M_n({\mcA})\hookrightarrow {\mcA}\otimes_\pi \mathcal{L}^p$ induces an isomorphisms
\item\label{ellptomtwo} $\mathcal{L}^p$-stability implies $M_n$-stability
\item\label{ntimesstab} $H(\iota)=n\theta^{M_n}$ if $H$ is $M_n$-stable and additive, with  $\iota:{\mcA}\to M_n ({\mcA}),a\mapsto 1_n\otimes a$ the canonical inclusion; similarly, $H(\iota)=n\theta^{\mc{L}^p}$ if we view $M_n ({\mcA})$ as a sitting inside ${\mcA}\otimes_\pi\mc{L}^p$.
\end{enumerate}
\end{Prop}
\begin{proof}
\ref{sum}: This follows as $H$ preserves direct sums:
$$H(A\oplus B)=H(A)\oplus H(B).$$
Therefore, $H(\Delta_A)=\Delta_{H(A)}$, where we denote by $\Delta :A\to A\oplus A, a\mapsto (a,a)$ the diagonal. The result follows from $\phi+\psi=\phi\oplus\psi\circ \Delta_A$  (compare, for example, \cite{MR2207702}, Proposition 3.1.2). \note{Some more details are in \cite{CuntzMeyer}, proof of Proposition 3.3, in the preliminary version.}

\ref{Mtwo} is also proved in \cite{MR2207702}: If $U\in {\mcA}$ is invertible, then \scalebox{0.7}{$\mb U& \\&1\me$} defines an automorphism of ${\mcA}\otimes M_2$. As the two natural inclusions of ${\mcA}$ in ${\mcA}\otimes M_2$ are the same under any $M_2$-stable functor, $U$ acts trivially.

If $\hsp$ is a separable Hilbert-space, then factoring the natural inclusion $\iota:M_n\to \mathcal{L}^p(\hsp)$ as the stabilisation $M_n\to M_n\otimes_\pi \mc{L}^p(\hsp)$ followed by the canonical identification $\mc{L}^p(\hsp^n)\approx \mc{L}^p(\hsp)$ - which both induce isomorphisms - we get \ref{naturalinclusion} in the scalar case; tensoring by ${\mcA}$ gives the general statement.

\ref{ellptomtwo} follows now from commutativity of 
\[\xymatrix{H(\mcA)\ar[dr]^{\theta_{\mcA}^{\mc{L}^p}}\ar[d]_{\theta_{\mcA}^{M_n}}\\
 H(M_n(\mcA))\ar[r]_{H(\iota\otimes \id_A)}&H(\mc{L}^p\potimes \mcA)
}\]

The first part of \ref{ntimesstab} follows from  \ref{Mtwo}, and entails the second.
\end{proof}

We restate the following fundamental result from \cite{MR2207702}, Theorem 4.2.1, in the form appropriate for later use (using Proposition \ref{functorgeneralities} -- \ref{ellptomtwo}):
\begin{Th}\label{diffotopyinvariance} Every functor from the category of locally convex algebras to the category of abelian groups which is split exact  and $\mc{L}^p$-stable for some $p\geq 2$ is diffotopy invariant.
\end{Th}

\section{Morita contexts and split exact functors}\label{contexts}
The following definition of Morita context is basically from \cite{MR2240217}, but we carry it over to arbitrary functors.  We also add some isomorphisms to the definitions in \cite{MR2240217} in order to make the existence of a Morita-context a weaker condition than being isomorphic.
\begin{Def}\label{contextcond}
Let $A$ and $B$ be locally convex algebras. Then a Morita-context from $A$ to $B$ is given by data $(\phi,D,\psi,\xi_i,\eta_i)$\index{$(\phi,D,\psi,\xi_i,\eta_i)$}, where $D$ is a locally convex algebra, $\phi:A\to D$, $\psi:B\to D$ are isomorphisms onto subalgebras of $D$, and sequences $\eta_i$, $\xi_i$ in $D$ such that 
\begin{enumerate}
\item $\eta_i \phi(A)\xi_j\in\psi(B)$ for all $i,j$
\item $(\eta_i \phi(a)\xi_j)_{ij}\in\sco\otimes_\pi \psi(B)$
\item\label{convergencecontext} $\sum \xi_i\eta_i \phi(a)=\phi(a)$ for all $a\in A$ (convergence in $\phi(A)$).
\end{enumerate}
\end{Def}
With this definition:
\begin{itemize}
\item If $\phi:A\to B$ is an isomorphism, then we get a Morita context $(\phi^+,B^+,\cdot^+)$ from $A$ to $B$, where $B^+$ is the unitization of $B$, $b\mapsto b^+$ the canonical embedding, and $\phi^+=\phi\circ\cdot^+$
\item In particular, there is now a canonical Morita-context from $A$ to $A$, for any locally convex algebra $A$
\item If $B\subseteq \sco\otimes_\pi C$ is a subalgebra, and we call a corner in $B$ a subalgebra $A\subseteq B$ of the form $\sum_{i=0}^k e_{0i}B\sum_{i=0}^k e_{i0}$, then there is a trivial context from $A$ to $B$
\item If $B$ is row and column-stable ($e_{0i}B,Be_{i0}\subseteq B$ for all $i$), then there is a context from $B$ to $A$, where $B$ is a subalgebra of $\sco\otimes_\pi A$.
\end{itemize}
 
\begin{Def} Let $H$ be a   functor on the category of locally convex algebras. Then we define $\theta:A\to \sco\otimes_\pi B,\,a\mapsto (\psi^{-1}(\eta_i\phi(a)\xi_j))_{ij}$ and $H(\phi,D,\psi,\xi_i,\eta_i):=H(\theta)$.
\end{Def}
Note that for $a,a'\in A$ we have 
\begin{align*}
(\psi^{-1}(\eta_i\phi(a)\xi_j))_{ij}(\psi^{-1}(\eta_k\phi(a')\xi_l))_{kl}=(\psi^{-1}(\eta_i\phi(a)\sum_{m}\xi_m\eta_m \phi(a')\xi_j))_{ij}
\end{align*}
which equals $\theta(aa')$ by \ref{convergencecontext} in the definition of a Morita context. Hence $\theta$ is indeed a homomorphism.
\begin{Def} A Morita-bicontext between locally convex algebras $A$ and $B$ is given by two Morita-contexts from $A$ to $B$ and $B$ to $A$ respectively, of the form $(\phi, D,\psi,\xi_i^A,\eta_i^A)$ and $(\psi,D,\phi,\xi_i^B,\eta_i^B)$ such that 
\begin{enumerate}
\item $\phi(A)\xi_i^A\xi_j^B\subseteq\phi(A)$, $\eta_i^B\eta_j^A\phi(A)\subseteq \phi(A)$ (left compatibility)
\item $\psi(B)\xi_i^B\xi_j^A\subseteq \psi(B)$, $\eta_i^A\eta_j^B\psi(B)\subseteq \psi(B)$ (right compatibility).
\end{enumerate}
\end{Def}
\begin{Th}\label{contweakeriso} Given two Morita contexts as in the above definition, and a diffotopy invariant, $\sco$-stable functor $H$, we have
$$H(\psi,D,\phi,\xi_i^B,\eta_i^B)\circ H(\phi, D,\psi,\xi_i^A,\eta_i^A)= \id_{H(A)}\text{ if they are left compatible }$$
$$ H(\phi, D,\psi,\xi_i^A,\eta_i^A)\circ H(\psi,D,\phi,\xi_i^B,\eta_i^B)=\id_{H(B)}\text{ if they are right compatible}$$
\end{Th}
\begin{proof}
The proof is an adaptation of the one from \cite{MR2240217} Lemma 7.2. Denote the isomorphism $H(A)\to H(\sco\potimes A)$ given by $\sco$-stability by $\ee_A$. Then more precisely, we have to show that
\begin{align*}
H(\phi,D,\psi,\xi_i^B,\eta_i^B)\circ \ee_B^{-1}\circ H(\phi, D,\psi,\xi_i^A,\eta_i^A)
\end{align*}
is invertible. We suppose left compatibility; then denoting $\theta^A$ and $\theta^B$ the maps $A\to \sco\potimes B$ and $B\to\sco\potimes A$ determined by the two contexts, and multiplying by $\ee_{\sco\potimes A}$  on the left, we see that  it suffices to show that the composition $(\sco\otimes\theta_B)\circ\theta_A$ induces an invertible map under $H$. Now this is the map
\begin{equation}\label{compo}a\mapsto (\phi^{-1}(\eta_i^B(\eta_k^A\phi(a)\xi_l^A)\xi_j^B))_{iklj}\end{equation}
and it is diffotopic to the stabilisation as follows. The $L\times L$ matrix with entries $\phi^{-1}(\hat\eta_\alpha(t)\phi(a)\hat\xi_\beta(t))$, $\alpha,\beta\in \mN^2\cup\{0\}$ with
\begin{align*}
\hat\xi_0(t)=\cos(t) 1\hspace{2cm} \hat\xi_{il}(t)=\sin(t)\xi_i^B\xi_l^A\\
\hat\eta_0(t)=\cos(t) 1\hspace{2cm}\hat\eta_{kj}=\sin(t) \eta_k^B\eta_j^A
\end{align*}
yields a diffotopy of the map in equation \ref{compo}.
\end{proof}

\section{Quasihomomorphisms and induced morphisms}

The notion of quasihomomorphisms was introduced in \cite{MR733641}, and used for example in \cite{MR2207702} to induce elements in $kk$. We modify it according to our needs.
\begin{Def} Let $\mcA$ and $ \mcB$ be locally convex algebras, $\hat\mcB$ an algebra containing $\mcB$ as a subalgebra. Then an $\LC$-quasihomomorphism from $\mcA$ to $\mcB$ (relative to $\hat{\mcB}$) is given by a
pair $(\alpha,\bar\alpha)$ of morphisms ${\mcA}\to \hat{\mcB}$ such that the maps
\begin{itemize}
\item ${\mcA}\to \hat{{{\mcB}}},a\mapsto \alpha(a)-\bar\alpha(a)$, 
\item ${\mcA}\times {{\mcB}}\to \hat{{{\mcB}}},\, (a,b)\mapsto \alpha(a)b$ and 
\item ${{\mcB}}\times{\mcA}\to \hat{{{\mcB}}},\, (b,a)\mapsto b\alpha(a)$
\end{itemize}
are all three actually ${{\mcB}}$-valued and continuous (with respect to the topology on $\mcA$ and  ${{\mcB}}$). We denote such a quasihomomorphism by $(\alpha,\bar\alpha):{\mcA}\rrarrow \hat{{{\mcB}}}\trianglerighteq {{\mcB}}$.

If $\phi:{\mcC}\to \mcA$ is a homomorphism of locally convex algebras, then  $(\alpha\circ\phi,\bar\alpha\circ \phi):{\mcC}\rrarrow \hat {\mcB}\trianglerighteq {\mcB}$ defines a quasihomomorphism which we will denote $(\alpha, \bar\alpha)\circ \phi$.
\end{Def}

Split exact functors have additional functoriality properties with respect to quasihomomorphisms. We first define the objects involved in the construction.
\begin{Lem}\label{newalgebra} Let $\mcA$ be a locally convex algebra,  $\hat {\mcB}$ an algebra and ${\mcB}$ a locally convex subalgebra of $\hat {\mcB}$, and $\alpha:\mcA\to \hat {\mcB}$ a homomorphism of   algebras such that the maps $\mcA\times {\mcB}\to\hat {\mcB}, (a,b)\mapsto \alpha(a)b$ and $ {\mcB}\times \mcA\to \hat {\mcB},\, (b,a)\mapsto b\alpha(a)$ are actually ${\mcB}$-valued and continuous (with respect to the topology on $\mcA$ and ${\mcB}$). Then the sum of topological vector space ${D}:=\mcA\oplus {\mcB}$, equipped with the multiplication 
$$(a,b)(a',b'):=(aa',\alpha(a)b'+b\alpha(a')+bb')$$
is a locally convex algebra $D_\alpha$\index{$D_\alpha$}. ${\mcB}$ identifies to an ideal in $D_\alpha$ that carries the subspace topology via the inclusion $\iota_{\mcB}:{\mcB}\to D_\alpha,\; b\mapsto (0,b)$, and $\mcA$ is a quotient of $D_\alpha$.

Further $D_\alpha$ fits into a split short exact sequence
\[\xymatrix{ 0\ar[r]&{\mcB}\ar[r]&D_\alpha\ar[r]&\mcA\ar[r]&0}\]
with split $\alpha':=\id_{\mcA}\oplus 0$.
\end{Lem}
\begin{proof} By hypothesis, the multiplication is continuous. Further, we may identify $D_\alpha$ with the subalgebra
$$\{(a,x)\in \mcA\times\hat\mcB\big| x-\alpha(a)\in B\}$$
of $\mcA\times\hat\mcB$, and hence the multiplication is associative.
\end{proof}
\begin{Rem} If ${\mcA},\hat {\mcB},{\mcB}$ and $\alpha$ come from an $\LC$-quasihomomorphism $(\alpha,\bar\alpha):{\mcA}\rrarrow \hat {\mcB}\trianglerighteq {\mcB}$, then  $\bar\alpha':=\id_{\mcA}\oplus (\bar\alpha-\alpha):{\mcA}\to D_\alpha$ defines another split of the above extension.\ifthenelse{\boolean{notes}}{This is an algebra-morphism as for all $a,a'\in {\mcA}$:
\begin{align*}&(a,\bar\alpha(a)-\alpha(a))(a',\bar\alpha(a')-\alpha(a'))\\ =&(aa',\alpha(a)\bar\alpha(a')-\alpha(aa')+\bar\alpha(a)\alpha(a')-\alpha(aa'))
						+(\bar\alpha(a)-\alpha(a))(\bar\alpha(a')-\alpha(a'))\\
					=&(aa',\bar\alpha(aa')-\alpha(aa')),
\end{align*}}{}
Further we see that it is actually possible to suppose that $\hat B$ is a locally convex algebra and $\alpha$, $\bar\alpha$ are morphisms of locally convex algebras. However, the elements we will use do not naturally have this form.
\end{Rem}

We will usually identify ${\mcB}$ with its image in $D$ and write $\alpha$ and $\bar\alpha$ instead of $\alpha'$, $\bar\alpha'$. 
\begin{Def}\label{quasiinduced} Let $(\alpha,\bar\alpha):{\mcA}\rrarrow \hat {\mcB}\trianglerighteq {\mcB}$ be an $\LC$-quasihomomorphism. Then for every split exact functor $H$ on the category of locally convex algebras the composition $H(\iota_{\mcB})^{-1}\circ (H(\alpha)-H(\bar\alpha'))$ defines a group homomorphism $H({\mcA})\to H({\mcB})$, denoted $H(\alpha,\bar\alpha)$ and called the homomorphism induced by the quasihomomorphism $(\alpha,\bar\alpha)$; if there is no risk of confusion, we write $H(\alpha,\bar\alpha)=(\alpha,\bar\alpha)_*$.
\end{Def}
\begin{Lem}\label{morphofdoublesplit} If $(\phi_1,\phi_2,\phi_3)$ is a morphism of double split extensions, i.e.:
\begin{equation}\label{homologycommutes}\xymatrix{ 0\ar[r]&{\mcB}\ar[d]_{\phi_1}\ar[r]^\iota &D\ar[d]_{\phi_2}\ar[r]|{\,\,\pi\,}&{\mcA}\ar@/_.3cm/[l]_{\alpha}\ar@/^.3cm/[l]^{\bar\alpha}\ar[d]^{\phi_3}\ar[r]&0\\
0\ar[r]&{\mcB}'\ar[r]^{\iota'}&D'\ar[r]|{\,\,\pi'\,}&{\mcA}'\ar@/_.3cm/[l]_{\beta}\ar@/^.3cm/[l]^{\bar\beta}\ar[r]&0
}\end{equation}
commutes in the usual sense and $\phi_2\circ\alpha=\beta\circ\phi_3$, $\phi_2\circ\bar\alpha=\bar\beta\circ\phi_3$, then for every split exact functor $H$
$$H(\phi_1)\circ H(\alpha,\bar\alpha)=H(\beta,\bar\beta)\circ H(\phi_3).$$
\end{Lem}
\begin{proof} As $H(\iota')$ is injective and
\begin{align*} H(\iota')\circ H(\phi_1)\circ H(\alpha,\bar\alpha)=&H(\phi_2)\circ (H(\alpha)-H(\bar\alpha))\\
=&(H(\beta)-H(\bar\beta))\circ H(\phi_3)\\
=&H(\iota')\circ H(\beta,\bar\beta)\circ H(\phi_3)
\end{align*}
the result follows.
\end{proof}
\begin{Rem}\label{remarkforinverse}
It suffices to suppose that we are given a commutative diagram such as the one in \ref{homologycommutes} that commutes after applying $H$.
\end{Rem}
\begin{Rem}\label{differentdalpha} In the setting of $C^*$-algebras, one usually applies a slightly different construction (e.g. \cite{MR2207702}). If we fix a $C^*$-quasihomomorphism $(\alpha,\bar\alpha):{\mcA}\rrarrow \mathcal{M}({\mcB})\trianglerighteq {\mcB}$ and denote by $\bar D_\alpha$  the subalgebra of ${\mcA}\oplus_{\cat{C}^*} \mathcal{M}({\mcB})$ generated by ${\mcB}$ and $(a,\alpha(a))$, then there is a unique morphism $\theta$ yielding a commutative diagram
\[\xymatrix{ 0\ar[r]&{\mcB}\ar@{=}[d]\ar[r]&D_\alpha\ar[r]\ar@{..>}^\theta[d]&{\mcA}\ar@/_.4cm/_{\id_{\mcA}\oplus 0}[l]\ar[r]\ar@{=}[d]&0\\
0\ar[r]&{\mcB}\ar[r]&\bar D_\alpha\ar[r]&{\mcA}\ar@/^.4cm/^{\id_{\mcA}\oplus \alpha}[l]\ar[r]&0}\]
That is, because the lefthand side commutes,  $\theta$ has to fix ${\mcB}$, and as the right hand side commutes, $(a,0)\mapsto (a,\alpha(a))$. It is easily calculated that the morphism thus determined actually is a $*$-homomorphism (if we equip $D_\alpha$ with the obvious involution).\ifthenelse{\boolean{notes}}{ In fact, we have:
\begin{align*}&\theta((a,b)\cdot_{D_\alpha}(a',b'))=\theta((aa',\alpha(a)b'+b\alpha(a')+bb'))\\
=&(aa',\alpha(aa')+\alpha(a)b'+b\alpha(a')+bb')=(a,\alpha(a)+b)\cdot_{\bar D_\alpha}(a',\alpha(a')+b').\\
\end{align*}}{} Therefore the construction coincides with the one in the $\cat{C}^*$-setting.
 
\end{Rem}

\begin{Prop}\label{quasirules} Let $(\alpha,\bar\alpha):{\mcA}\rrarrow \hat {\mcB}\trianglerighteq {\mcB}$ be an $\LC$-quasihomomorphism, $H$ a split exact functor  on the category of locally convex algebras with values in abelian groups. Then the following properties hold:
\begin{enumerate}
\item\label{quasiprecompose} $H((\alpha,\bar\alpha)\circ\phi)=H(\alpha,\bar\alpha)\circ H(\phi)$ for every morphism $\phi:{\mcC}\to {\mcA}$,
\item\label{quasipostcompose} $H(\psi\circ(\alpha,\bar\alpha))=H(\psi)\circ H(\alpha,\bar\alpha)$ for every homomorphism $\psi:\hat {\mcB}\to \hat {\mcC}$ of  algebras, such that there is a locally convex subalgebra  ${\mcC}\subseteq\hat {\mcC}$ such that $(\psi\circ\alpha,\psi\circ\bar\alpha):{\mcA}\rrarrow \hat {\mcC}\trianglerighteq {\mcC}$ is a quasihomomorphism,
\note{\item\label{onemoreprop} if $(\beta,\bar\beta):\mcA'\rrarrow \hat{\mcB'}\trianglerighteq \mcB'$ is a quasihomomorphism, and $\phi:\mcA\to\mcA'$, $\psi:\mcB\to \mcB'$ are homomorphisms of locally convex algebras such that $\psi\circ (\alpha-\bar\alpha)=(\beta-\bar\beta)\circ \phi$, then
$$ H(\psi)\circ H(\alpha,\bar\alpha)=H(\beta,\bar\beta)\circ H(\phi),$$}{}
\item\label{negativequasi} $H(\alpha,\bar\alpha)=-H(\bar\alpha,\alpha)$,
\item\label{orthosum} if $\alpha-\bar\alpha$ is a homomorphism orthogonal to $\bar\alpha$, then $H(\alpha,\bar\alpha)=H(\alpha-\bar\alpha)$.
\end{enumerate}
\end{Prop}
\begin{proof} For \ref{quasiprecompose}, just note that $\phi\oplus\id_{\mcB}:{\mcC}\oplus {\mcB}\to {\mcA}\oplus {\mcB}$ is actually a morphism of locally convex algebras $D_{\alpha\circ\phi}\to D_\alpha$ such that $(\id_{\mcB},\phi\oplus\id_{\mcB},\phi)$ is a morphism of the double split exact sequences associated to $D_{\alpha\circ\phi}$ and $D_\alpha$, and apply Lemma \ref{morphofdoublesplit}.

Similar reasoning applies to the morphism $(\psi,\psi\oplus \id_{\mcA},\id_{\mcA})$ and proves \ref{quasipostcompose}. 

\note{\ref{onemoreprop} follows by observing that (eventually replacing by $D_\alpha$ and $D_\beta$) we may assume $\psi$ extends to a homomorphism $\hat \mcB\to \hat{\mcB'}$ and that $\psi\circ\alpha=\beta\circ\phi$, $\psi\circ\bar\alpha=\bar\beta\circ\phi$. Thus we get a commutative diagram:
\[\xymatrix{ 0\ar[r]& {\mcB}\ar[d]_{\phi}\ar[r] &D_{\alpha}\ar[d]_{\psi\oplus\phi}\ar[r]&{\mcA}\ar@/_.3cm/[l]_{\alpha'}\ar@/^.3cm/[l]^{\bar\alpha'}\ar[d]^{\psi}\ar[r]&0\\
0\ar[r]&{\mcB'}\ar[r]&D_{\beta}\ar[r]&{\mcA'} \ar@/_.3cm/[l]_{\beta'}\ar@/^.3cm/[l]^{\bar\beta'}\ar[r]&0.
}\]}{}

For \ref{negativequasi}, use that with $\phi_2:D_\alpha\to D_{\bar\alpha},(a,b)\mapsto (a,\alpha(a)-\bar\alpha(a)+b)$ we have a morphism $(\id_{{\mcB}_*},{\phi_2}_*,-\id_{{\mcA}_*})$ \underline{after} applying $H$ and use Remark \ref{remarkforinverse}.

\ref{orthosum} follows from Proposition \ref{functorgeneralities} - \ref{sum}.\note{ Namely: 
\begin{align*} &H(\alpha,\bar\alpha)=H(\iota)^{-}\circ(H(\alpha)-H(\bar\alpha))\\
=&H(\iota)^-\circ (H(\alpha-\bar\alpha)+H(\bar\alpha)-H(\bar\alpha))=H(\iota)^-\circ H(\alpha-\bar\alpha).
\end{align*}}{}
\end{proof}
\begin{Prop}
Let

\[\xymatrix{&&0\ar[dd]&&0\ar[dd]&&0\ar[dd]\\ \\
0\ar[rr]&&I^1_{\;\;1}\ar[dd]_{\phi^{12}_{\;\;\;1}}\ar[rr]^{\phi^1_{\;\;12}}&&I^1_{\;\;2}\ar[dd]_{\phi^{12}_{\;\;\;2}}\ar[rr]|{\;\phi^1_{\;\;23}}&&I^1_{\;\;3}\ar@/^.4cm/[ll]^{\sigma^{1}_{\;\;23}}\ar@/_.4cm/[ll]_{\bar\sigma^{1}_{\;\;23}}\ar[dd]_{\phi^{12}_{\;\;\;3}}\ar[rr]&&0\\ \\
0\ar[rr]&&I^2_{\;\;1}\ar[dd]|{\;\phi^{23}_{\;\;\;1}\,}\ar[rr]^{\phi^2_{\;\;12}}&&I^2_{\;\;2}\ar[dd]|{\;\phi^{23}_{\;\;\;2}\,}\ar[rr]|{\;\phi^2_{\;\;23}\,}&&I^2_{\;\;3}\ar@/^.4cm/[ll]^{\bar\sigma^{2}_{\;\;23}}\ar@/_.4cm/[ll]_{\sigma^{2}_{\;\;23}}\ar[dd]|{\phi^{23}_{\;\;\;3}}\ar[rr]&&0\\ \\
0\ar[rr]&&I^3_{\;\;1}\ar[dd]\ar[rr]^{\phi^3_{\;\;12}}\ar@/^.4cm/[uu]^{\sigma^{23}_{\;\;\;1}}\ar@/_.4cm/[uu]_{\bar\sigma^{23}_{\;\;\;1}}&&I^3_{\;\;2}\ar[dd]\ar[rr]|{\;\phi^3_{\;\;23}\,}\ar@/^.4cm/[uu]^{\sigma^{23}_{\;\;\;2}}\ar@/_.4cm/[uu]_{\bar\sigma^{23}_{\;\;\;2}}&&I^3_{\;\;3}\ar[dd]\ar[rr]\ar@/^.4cm/[uu]^{\bar\sigma^{23}_{\;\;\;3}}\ar@/_.4cm/[uu]_{\sigma^{23}_{\;\;\;3}}\ar@/^.4cm/[ll]^{\sigma^{3}_{\;\;23}}\ar@/_.4cm/[ll]_{\bar\sigma^{3}_{\;\;23}}&&0\\ \\
&&0&&0&&0}\]
be a diagram that is row- and column-wise given by double-split short exact sequences of locally convex algebras and morphisms. We suppose that all squares involving only morphisms of type $\phi$ commute, and that $\phi^2_{\;\;12}\circ\sigma^{23}_{\;\;\;1}=\sigma^{23}_{\;\;\;2}\circ\phi^3_{\;\;12}$,  $\sigma^2_{\;\;23}\circ\phi^1_{\;\;23}=\phi^{12}_{\;\;\;2}\circ\sigma^1_{\;\;23}$, $\sigma^{23}_{\;\;\;2}\circ \sigma^3_{\;\;23}=\sigma^2_{\;\;23}\circ\sigma^{23}_{\;\;\;3}$ and the same relations with all $\sigma$ replaced by $\bar\sigma$; finally $\bar\sigma^{23}_{\;\;\;2}\circ \sigma^3_{\;\;23}=\sigma^2_{\;\;23}\circ\bar\sigma^{23}_{\;\;\;3}$ and $\sigma^{23}_{\;\;\;2}\circ
\bar\sigma^3_{\;\;23}=\bar\sigma^2_{\;\;23}\circ\sigma^{23}_{\;\;\;3}$. Then for every split exact functor $H$ we have
$$H(\sigma^{23}_{\;\;\;1},\bar\sigma^{23}_{\;\;\;1})\circ H(\sigma^3_{\;\;23},\bar\sigma^3_{\;\;23})=H(\sigma^1_{\;\;23},\bar\sigma^1_{\;\,23})\circ H(\sigma^{23}_{\;\;\;3},\bar\sigma^{23}_{\;\;\;3}),$$
in other words, the composition of the morphisms induced by the outer edges in the diagram are the same.
\end{Prop}
\begin{proof} As 
$$H(\phi^2_{\;\;12})\circ H(\phi^{12}_{\;\;\;1})=H(\phi^{12}_{\;\;\;2})\circ H(\phi^1_{\;\;12})$$ 
and both sides are injective, it suffices to show
\begin{align*}& H(\phi^2_{\;\;12})\circ H(\phi^{12}_{\;\;\;1})\circ H(\sigma^{23}_{\;\;\;1},\bar\sigma^{23}_{\;\;\;1})\circ H(\sigma^3_{\;\;23},\bar\sigma^3_{\;\;23})\\
=&H(\phi^{12}_{\;\;\;2})\circ H(\phi^1_{\;\;12})\circ H(\sigma^1_{\;\;23},\bar\sigma^1_{\;\,23})\circ H(\sigma^{23}_{\;\;\;3},\bar\sigma^{23}_{\;\;\;3}).
\end{align*}
This follows from the hypotheses by an easy diagram chase.
\end{proof}

\begin{Def}\label{outerprods} We will call a quasihomomorphism  $(\alpha,\bar\alpha):{\mcA}\rrarrow \hat {\mcB}\trianglerighteq {\mcB}$ a quasihomomorphism in standard form if $\hat\mcB$ is isomorphic to $D_\alpha$.

Given a quasihomomorphism in standard form as above and a supplementary locally convex algebra ${\mcC}$, we define the quasihomomorphisms $(\alpha,\bar\alpha)\otimes {\mcC}:=(\alpha\otimes {\mcC},\bar\alpha\otimes {\mcC})$, and similarly ${\mcC}\otimes (\alpha,\bar\alpha)$.

If $\mcA$ and $\mcB$ are locally convex algebras, we define
$$\sigma_{\mcA,\mcB}:\mcA\potimes\mcB\to\mcB\potimes\mcA,\; a\otimes b\mapsto b\otimes a.$$
\end{Def}
Note that as the projective tensor product preserves split exactness, tensoring by ${\mcC}$ yields again a quasihomomorphism, which is furthermore again in standard form.

If we start with an arbitrary quasihomomorphism $(\alpha,\bar\alpha):{\mcA}\rrarrow\hat {\mcB}\trianglerighteq {\mcB}$, then we may  replace $\hat {\mcB}$ by $D_\alpha$ to obtain a quasihomomorphism in standard form.

Now let $(\alpha_i,\bar\alpha_i):{\mcA}_i\rrarrow \hat {\mcB}_i\trianglerighteq {\mcB}_i$ be two $\LC$-quasihomomorphisms. Then we have associated exact sequences $E_i:=\big(0\to {\mcB}_i\to D_i\to {\mcA}_i\to 0\big)$ as in Lemma \ref{newalgebra} and double splits $\alpha_i'$, $\bar\alpha_i'$. As the projective tensor product preserves exactness of linearly split extensions, we obtain double split exact sequences containing $E_1\otimes_\pi {\mcB}_2$, $E_1\otimes_\pi D_2$, $E_1\otimes_\pi {\mcA}_2$, and similarly ${\mcB}_1\otimes_\pi E_2$, $D_1\otimes_\pi E_2$ and ${\mcA}_1\otimes_\pi E_2$ that fit into a diagram as in the above Lemma, and obviously satisfy the commutativity conditions. Identifying $D_1\otimes\hat {\mcC}$ with the algebra associated to the quasihomomorphism $(\alpha_1,\bar\alpha_1)\otimes {\mcC}:{\mcA}_1\otimes_\pi {\mcC}\rrarrow{\mcB}_1\otimes_\pi {\mcC}\trianglerighteq {\mcB}_1\otimes_\pi {\mcC}$ for any locally convex algebra ${\mcC}$, and similarly for the others, we have proved most of the following
\begin{Prop}\label{doublesplitouterproduct}
$$H({\mcB}_1\otimes(\alpha_2,\bar\alpha_2))\circ H((\alpha_1,\bar\alpha_1)\otimes {\mcA}_2)=H((\alpha_1,\bar\alpha_1)\otimes {\mcB}_2)\circ H({\mcA}_1\otimes (\alpha_2,\bar\alpha_2)),$$
or, in other words, the obvious outer product of quasihomomorphisms  is compatible with the induction process. In particular
$$H({\mcB}_1\otimes \phi)\circ H((\alpha_1,\bar\alpha_1)\otimes {\mcA}_2)=H((\alpha_1,\bar\alpha_1)\otimes {\mcB}_2)\circ H({\mcA}_1\otimes \phi),$$
for every homomorphism $\phi: {\mcA}_2\to {\mcB}_2$.

If $(\alpha,\bar\alpha):{\mcA}\rrarrow \hat {\mcB}\trianglerighteq {\mcB}$ is a quasihomomorphism in standard form and ${\mcC}$ a supplementary locally convex algebra, then
$$H((\alpha,\bar\alpha)\otimes {\mcC})=H(\sigma_{{\mcC},{\mcB}})\circ H({\mcC}\otimes(\alpha,\bar\alpha))\circ H(\sigma_{{\mcA},{\mcC}}).$$
\end{Prop}
\begin{proof} The last statement follows by identifying $D_{\alpha\otimes {\mcC}}\approx D_\alpha\otimes_\pi {\mcC}$, applying Lemma \ref{morphofdoublesplit} to
\[\xymatrix{ 0\ar[r]&{\mcB}\otimes_\pi {\mcC}\ar[d]_{\sigma_{{\mcB},{\mcC}}}\ar[r] &D_{\alpha\otimes {\mcC}}\ar[d]_{\sigma_{D_\alpha,{\mcC}}}\ar[r]&{\mcA}\otimes_\pi {\mcC}\ar@/_.3cm/[l]_{\alpha\otimes {\mcC}}\ar@/^.3cm/[l]^{\bar\alpha\otimes {\mcC}}\ar[d]^{\sigma_{{\mcA},{\mcC}}}\ar[r]&0\\
0\ar[r]&{\mcC}\otimes_\pi {\mcB}\ar[r]&D_{{\mcC}\otimes\alpha}\ar[r]&{\mcC}\otimes_\pi {\mcA}\ar@/_.3cm/[l]_{{\mcC}\otimes\alpha}\ar@/^.3cm/[l]^{{\mcC}\otimes\bar\alpha}\ar[r]&0
}\]
and multiplying the result by ${\sigma_{{\mcB},{\mcC}}}_*^{-1}={\sigma_{{\mcC},{\mcB}}}_*$ on the left.
\end{proof}
In particular, for suitably nice ideals and algebras, it is easy to see that for "unbounded $({\mcA}_i,{\mcB}_i)$ cycles" $[D_i]$ we get $Qh([D_1]\potimes {\mcB}_2)=Qh([D_1])\potimes {\mcB}_2$. Hence we may apply the above to obtain
$$H({\mcB}_1\otimes_\pi [D_2])\circ H([D_1]\otimes_\pi {\mcA}_2)=H([D_1]\otimes_\pi {\mcB}_2)\circ H({\mcA}_1\otimes_\pi  [D_2]).$$

We will make use of the following result, which shows that the induction process is automatically compatible with stabilisation for stable functors,  later on:
\begin{Prop}\label{explicitstab} Let $H$ be a $J$-stable, split exact functor. Then for every quasihomomorphism $(\alpha,\bar\alpha)$ from ${\mcA}$ to ${\mcB}$ the following diagram commutes:
\[\xymatrix{H({\mcA})\ar[d]_{\theta^J_{\mcA}}\ar[rr]^{H(\alpha,\bar\alpha)}&&H({\mcB})\ar[d]^{\theta^J_{\mcB}}\\
H({\mcA}\otimes_\pi J)\ar[rr]_{H((\alpha,\bar\alpha)\otimes J)}&&H({\mcB}\otimes_\pi J)
}\]
\end{Prop}
\begin{proof} Identify
$D_{\alpha\otimes J}\approx D_\alpha\otimes_\pi J$ and apply Lemma \ref{morphofdoublesplit} to 
\[\xymatrix{ 0\ar[r]&{\mcB}\ar[d]_{\iota_{\mcB}}\ar[r] &D_{\alpha}\ar[d]_{\iota_{D_{\alpha}}}\ar[r]&{\mcA}\ar@/_.3cm/[l]_{\alpha}\ar@/^.3cm/[l]^{\bar\alpha}\ar[d]^{\iota_{\mcA}}\ar[r]&0\\
0\ar[r]&{\mcB}\otimes_\pi J\ar[r]&D_{\alpha}\otimes_\pi J\ar[r]&{\mcA}\otimes_\pi J\ar@/_.3cm/[l]_{\alpha\otimes J}\ar@/^.3cm/[l]^{\bar\alpha\otimes J}\ar[r]&0
}\]
where the maps $\iota$ are inducing the stabilisation isomorphism under $H$ and make the diagram commute.
\end{proof} 
Let $\hat \mcB$ be an algebra containing $\scy \mcB$ as an ideal. Set $x:[0,1]\to \mR,\; t\mapsto t$, and $\sCo_i\mcB:=\{f\in\scy\mcB| f(i)=0\}$. Then $(x-i)\sCo_i\mcB=\sCo_i\mcB$. Hence, if $(\alpha,\bar\alpha):\mcA\rrarrow \hat\mcB\trianglerighteq \scy\mcB$ is a quasihomomorphism, we get quasihomomorphisms
$$(\alpha_i,\bar\alpha_i):\mcA\rrarrow \hat B/\sCo_i\mcB\trianglerighteq B.$$
\begin{Def} A diffotopy is a quasihomomorphism $(\alpha,\bar\alpha):{\mcA}\rrarrow \hat{\mcB}\trianglerighteq \scy {\mcB}$ \end{Def}
We have:
\begin{Lem}\label{inductionvsdiffo} If $H$ is a diffotopy invariant, split exact functor and $(\alpha,\bar\alpha):{\mcA}\rrarrow \hat {\mcB}\trianglerighteq \scy {\mcB}$  a diffotopy with $\scy B$ and ideal in $\hat\mcB$, then
$$H(\alpha_0,\bar\alpha_0)=  H(\alpha_1,\bar\alpha_1).$$
\end{Lem}
\begin{proof} Denote $\pi_i:\hat \mcB\to \hat\mcB/\sCo_i\mcB$ the quotient maps, and apply Lemma \ref{morphofdoublesplit} to
\[\xymatrix{ 0\ar[r]&\scy {\mcB}\ar[d]_{\pi_i}\ar[r] &D_{\alpha}\ar[d]_{\id_{\mcA}\oplus \pi_i}\ar[r]&{\mcA}\ar@/_.3cm/[l]_{\alpha'}\ar@/^.3cm/[l]^{\bar\alpha'}\ar[d]^{\id_{\mcA}}\ar[r]&0\\
0\ar[r]&{\mcB}\ar[r]&D_{\alpha_i}\ar[r]&{\mcA} \ar@/_.3cm/[l]_{\alpha_i'}\ar@/^.3cm/[l]^{\bar\alpha_i'}\ar[r]&0.
}\]
\end{proof}

\subsection{The action of $K$-theory}\label{Kaction}
In this section, $K(\mcA)$ denotes the algebraic $K$-theory of an algebra $\mcA$, $\cup$ the outer product $K(\mcA)\otimes K(\mcB)\to K(\mcA\potimes\mcB)$. The notion of invariance under inner automorphisms was given in Definition \ref{functorprops}.

\begin{Prop}\label{pairing} Let $H$ be a split exact, $M_2$-stable functor from locally convex algebras to abelian groups. For every locally convex algebras $\mcA$  and $\mcB$ there is a product:
$$K(\mcA)\otimes H(\mcB)\to H(\mcA\potimes \mcB),\; x\otimes b\mapsto x\cdot b.$$
If $\mcA=\mC$ and $x=1\in K(\mcA)$, then 
\begin{equation}\label{unity} x\cdot b=b.\end{equation}
Furthermore, if $f:\mcA_1\to \mcA_2$, $g:\mcB_1\to \mcB_2$ are homomorphisms, $x_1\in K(\mcA_1)$, $b_1\in H(\mcB_1)$:
\begin{equation}\label{outercomp}f_*(x_1)\cdot g_*(b_1)=(f\otimes g)_*(x_1\cdot b_1).\end{equation}
If $\mcC$ is another locally convex algebra, $x\in K(\mcA)$, $y\in K(\mcB)$ and $c\in H(\mcC)$
$$(x\cup y)\cdot c=x\cdot (y\cdot c).$$
\end{Prop}
\begin{proof} By Proposition \ref{functorgeneralities}, $H$ is invariant under inner automorphisms and $M_n$-stable for all $n$. 

Assume first that $\mcA$ is unital. Let $p$ be an idempotent in $M_n(\mcA)$. Then define a morphism
$$\phi_p^{\mcB}:\mcB\to M_n(\mcA\potimes \mcB),\; b\mapsto p\otimes b.$$
This yields (by the universal property of the enveloping group) a homomorphism of groups
$$F_{\mcA}^{\mcB}:K(\mcA)\to \Hom(H(\mcB),H(\mcA\potimes \mcB)),\; [p]\mapsto H(\phi_p^{\mcB}),$$
which defines the product $x\cdot b:=F_{\mcA}^{\mcB}(x)(b)$. The compatibility with unital morphisms is straightforward.\note{ 	More precisely: $\mcA_1$, $\mcA_2$, $\mcB_1$ and $\mcB_2$ be unital, $f:\mcA_1\to \mcA_2$ be a unital homomorphism. Then we get for $[p]\in K(\mcA_1)$:
$$f_*([p])\cdot b_1=H((f\otimes \id)\circ \phi_p^{B_1})(b_1)=(f\otimes\id)_*([p]\cdot b_1),$$
thus the first equation in \ref{outercomp}; the proof of the second is similar.}{}
If $\mcC$ is another unital algebra, then for all $ x\in K(\mcA)$, $y\in K(\mcB)$ we have
$$F_{\mcA}^{\mcB\potimes \mcC}(x)\circ F_{\mcB}^{\mcC}(y)=F_{\mcA\potimes\mcB}^{\mcC}(x\cup y)$$
directly from the definitions.

To treat the nonunital case, it suffices to apply the following Lemma.
\end{proof}
\begin{Lem} If  
\[\xymatrix{0\ar[r]& J\ar[r]&\mcA_1\ar[r]^\pi&\mcA_2\ar[r]&0}\]
is a split exact sequence of locally convex algebras with $\mcA_1$ and $\mcA_2$ unital, then  $K(J)H(\mcB)\subseteq H(J\otimes \mcB)$.
\end{Lem}
\begin{proof} Let $x\in K(J)\subseteq K(\mcA_1)$, then
$$(\pi\otimes\id)_*(x\cdot b)=\pi_*(x)\cdot b=0.$$
\end{proof}
\note{To finish the proof of the Proposition, note that if $\mcA$ and $\mcB$ are nonunital algebras, then for $x\in K(\mcA)\subseteq K(\mcA^+)$, $y\in K(\mcB)\subseteq K(\mcB^+)$ we have $(x\cup y) c= x\cdot(y\cdot c)$ for $x$ and $y$ as elements of the unitalisations. Then because $H(\mcA\potimes\mcB\potimes\mcC)\subseteq H(\mcA^+\potimes\mcB^+\potimes\mcC)$, and the product is actually in the first group by compatibility with homomorphisms, we get equality in the first group.

It remains to show compatibility with morphisms. We have a diagram
\[\xymatrix{0\ar[r]& \mcA_1\ar[r]\ar[d]^f&\mcA_1^+\ar[r]^{\pi_1}\ar[d]^{f^+}&\mC\ar[r]\ar[d]^{\id_{\mC}}&0\\
						0\ar[r]&\mcA_2\ar[r]&\mcA_2^+\ar[r]^{\pi_2}&\mC\ar[r]&0}\]
	In particular, $\pi_2\circ f^+=\pi_1$. Let $x\in K(\mcA_1)\subseteq K(\mcA_1^+)$ and $b\in H(\mcB)$, then
	$$(\pi_2\otimes 1)_*(f_*(x)\cdot b)=(\pi_2\circ f\otimes 1)_*(x\cdot b)={\pi_1}_*(x)\cdot b)=0$$
	hence $f_*(x)\cdot b)\in K(\mcA_2\potimes \mcB)$, and similarly for $(f\otimes 1)_*(x\cdot b)$.}{}
\begin{Rem} Properties \ref{unity} and \ref{outercomp} classify the pairing, because if $p\in M_n(\mcA)$, $f:\mC\to M_n(\mcA)$ the homomorphism $\lambda\mapsto \lambda p$, then
$$[p]\cdot b=[f(1)]\cdot b=H(f\otimes\id_{\mcB})(1\cdot b)=H(f\otimes \id_{\mcB})(b).$$
\end{Rem}
	
It remains to compare the two ways a given quasihomomorphism acts on $H$:
\begin{Prop}\label{twoinduced}Let $(\alpha,\bar\alpha):\mC\rrarrow\hat \mcA\trianglerighteq \mcA$ be a quasihomomorphism and  $H$ a split exact, $M_2$-stable functor. Then for every locally convex algebra $\mcB$
\begin{enumerate}
\item the map $H(\mcB)\to H(\mcA\potimes\mcB)$ given as multiplication with $K(\alpha,\bar\alpha)(1)\in K(\mcA)$ and
\item the map $H((\alpha,\bar\alpha)\otimes \mcB)=H(\iota)^{-1}\circ(H(\alpha\otimes\id_{\mcB})-H(\bar\alpha\otimes \id_{\mcB}))$ from Definition \ref{quasiinduced}
\end{enumerate}
coincide.
\end{Prop}
\begin{proof}
We may suppose, replacing by $D_\alpha$, that $\hat\mcA$ is locally convex and $\mcA$ a closed ideal in it. Then for $\mcB=\mC$ this follows directly from the definitions. The general case follows from the scalar case by applying it to the functor $H^\mcB:=H(\,\cdot\,\otimes\mcB)$.
\end{proof}
This shows that quasihomomorphisms from $\mC$ to $B$ are determined by their action on $K$-theory:
\begin{Lem}\label{comparison} Let $(\alpha_i,\bar\alpha_i):\mC\rrarrow \hat\mcB\trianglerighteq \mcB$ be two quasihomomorphisms. Then $H(\alpha_1,\bar\alpha_1)=H(\alpha_2,\bar\alpha_2)$ for all split exact $M_2$ stable functors if and only if the equality holds for the specific functor $K$.
\end{Lem}
We also have the following useful
\begin{Lem}\label{inductionseq} Let $(\alpha,\bar\alpha):\mC\rrarrow\hat\mcB\trianglerighteq\mcB$ be a quasihomomorphism. Then there is a quasihomomorphism $(\beta,\bar\beta):\mC\rrarrow M_2(\mcB^+)\trianglerighteq M_2(\mcB)$ such that
$$H(\theta_B)\circ H(\alpha,\bar\alpha)=H(\alpha',\bar\alpha')$$
for every split exact $M_2$-stable functor, where $\theta_{\mcB}:\mcB\to M_2(\mcB)$ denotes the stabilisation.
\end{Lem}
More precisely, there is a quasihomomorphism $(\beta,\bar\beta):\mC\rrarrow M_2(\mcB^+)\trianglerighteq M_2(\mcB)$ that is "stably inner equivalent" to $(\alpha,\bar\alpha)$.
\begin{proof}
We identify a homomorphism $\phi:\mC\to\hat\mcB$ with the idempotent $\phi(1)$. By Proposition \ref{functorgeneralities}, $H$ is invariant under inner automorphisms. Assume further, eventually replacing by $D_\alpha$, that $\hat\mcB$ is a locally convex algebra, that $\mcB$ is an ideal in $\hat\mcB$ and that $\hat\mcB$ is unital. We identify $\mcB^+$ with the subalgebra of $\hat\mcB$ generated by $\mcB$ and $1$.  Set $p:=\alpha(1)$, $\bar p:=\bar\alpha(1)$. If
\[q:=\mb p&\\&1-p\me,\;\;\bar q:=\mb \bar p&\\&1-p\me\text{ and }V:=\mb p&1-p\\1-p&p\me,\]
then the quasihomomorphisms $(p\oplus 0,\bar p\oplus 0)$ and $(q,\bar q)$ induce the same maps under any split exact  functor. Further $V^2=1$ and 
$$VqV=\mb 1&\\&0\me\in M_2(\mcB^+)\text{ and }V\bar q V=V(\bar q-q)V+VqV\in M_2(\mcB^+).$$
We may thus set $(\alpha',\bar\alpha'):=(VqV,V\bar qV)$.
\end{proof}

\section{$\LC$-Kasparov modules and abstract Kasparov modules}
We will now introduce abstract Kasparov modules, an intermediate step between Kasparov-modules and quasihomomorphisms. There are several versions of this notion of abstract Kasparov module, compare \cite{MR2207702} and \cite{CuntzMeyer}.
In the latter, constructions for abstract Kasparov modules (in the $C^*$-setting) are done by simply lifting an abstract Kasparov module back to a classical one, and then transferring the construction back to the abstract setting (compare for example \cite{CuntzMeyer}, Lemma 8.33). For the product, this leads to the notion of double Kasparov module, and as noted again in \cite{CuntzMeyer}, the conditions in the non-$C^*$-setting get extremely technical, and it is not clear how useful this construction is for bornological algebras. 

We will describe the product in a different manner below, similar to the approach we sketched for $KK$-theory. Again, we have to modify the definitions to make them applicable to our settings.

\begin{Def} Let $\mcA,\mcB$ be locally convex algebras, $\hat \mcB$ a unital algebra and $\mcB\subseteq \hat \mcB$. An abstract Kasparov module from $\mcA$ to $\mcB$ with respect to $\hat \mcB$ is a triple $(\alpha, \bar \alpha,U)$ such that $U\in \hat \mcB$ is invertible, $\alpha,\bar\alpha:\mcA\to \hat \mcB$ are two homomorphisms and the map
$$\mcA\to \hat \mcB,\;a\mapsto \alpha(a)-U^{-1}\hat\alpha(a) U$$
is actually $\mcB$-valued and continuous. We then define the associated quasihomomorphism \index{$Qh(\alpha,\hat\alpha,U)$}$Qh(\alpha,\hat\alpha,U):=(\alpha,U^{-1}\bar\alpha U)$.\note{Meyer supposed the continuity in \cite{CuntzMeyer}, Cuntz seems to have forgotten it in the OWO-report.}{}
\end{Def}
 
We proceed to define an appropriate notion of locally convex bimodule:
\begin{Def} Let $\mc{A}$ and $\mc{B}$ be locally convex algebras. An $\LC$-Kasparov module from $\mcA$ to $\mc{B}$ is given by an algebra $\hat \mcB$ containing $\mcB$, elements $\ee,F\in\hat{\mc{B}}$ and an algebra morphism $\phi:\mcA\to\hat\mcB$  such that $\ee^2=1$, $F\ee=-\ee F$, $[\phi(a),\ee]=0$ and such that the maps
\begin{enumerate}
\item  $\mcB\to \hat\mcB,\; b\mapsto bF,Fb,\ee b,b\ee$
\item\label{continmult} $\mcA\otimes \mcB\to \hat\mcB,\;(a,b)\mapsto \phi(a)b,b\phi(a)$
\item $\mc{A}\to\hat\mcB,\; a\mapsto \phi(a)(1-F^2),(1-F^2)\phi(a)$
\item $\mc{A}\to\hat\mcB,\; [\phi(a),F]$ 
\end{enumerate}
are actually $\mcB$-valued and continuous. We then call $\ee$ a grading and denote the $\LC$-Kasparov $(\mcA,\mcB)$-module by $(\hat\mcB,\phi,F)$\index{$(\hat\mcB,\phi,F)$}.
\end{Def}
Note that it suffices to suppose the map $a\mapsto \phi(a)(1-F^2)$  be continuous $\mcB$-valued and $(1-F^2)\phi(a)\in\mcB$, continuity of the second in $a$ then follows; similarly for the maps $a\mapsto \phi(a)b$ and $a\mapsto b\phi(a)$in  \ref{continmult}.   There is also an obvious notion of representation of $\LC$-Kasparov modules which we will use in the sequel. We also identify $\LC$-Kasparov modules that are isomorphic in the obvious sense.
\begin{Prop} For a given $\LC$-Kasparov $(\mcA,\mcB)$-module $(\hat \mcB,\phi,F)$ with grading $\ee$ we set
\index{$W_F$}$W_F:=\ee (1-F^2)+P_{\ee} F+P_{\ee}^{\perp} F(2-F^2)$, where $P_\ee:=\nicefrac{1}{2}(1+\ee)$ and $P^\perp_\ee:=1-P_\ee$. This defines an associated abstract Kasparov module $AKM(\hat \mcB,\phi,F):=(P_{\ee}\phi,P_{\ee}^\perp \phi,W_F)$.
\end{Prop}
\note{\marginpar{W F invertible, do alg. case for KM}}{}
\begin{proof}
$W_F$ is it's own inverse:
\begin{align*}W_F^2=(1-F^2)^2+(P_{\ee}+P_{\ee}^\perp) F^2(2-F^2)=1,
\end{align*} 
and  modulo a continuous $\mcB$-valued map we have
\begin{equation}\label{AKM}P_\varepsilon\phi(a) W_F-W_F P_\varepsilon^\perp\phi(a)=P_{\ee}[\phi(a),F]-\phi(a)(1-F^2)\in J.\end{equation}
As follows from the hypothesis, multiplication by $W_F$ is continuous on $\mcB$, thus the result follows by multiplying equation \ref{AKM} by $W_F$ on the left. \ \note{ One really gets the result $$P_\ee[a,F]+P_\ee (a (1-F^2)+(1-F^2)a).$$ However, the result is shorter written like this.}
\end{proof}
We may thus give the following 
\begin{Def} For every $\LC$-Kasparov module $(\hat B,\phi,F)$ we define
$$\index{$Qh(\hat \mcB,\phi,F$)} Qh(\hat\mcB,\phi,F):=Qh(AKM(\hat\mcB,\phi,F)),$$
and if $H$ is a split exact functor
$$H(\hat\mcB,\phi,F):=H(Qh(\hat\mcB,\phi,F))=H(Qh(AKM(\hat\mcB,\phi,F))).$$
\end{Def}
The point of this definition is the following: Using the usual transformation $F\mapsto U_F$ (see below) is adequate in the setting of $p$-summable Fredholm modules, but it is not, for example, on smooth functions. 
\begin{Prop}\label{makeinvertible} Let $(E,\phi,F)$ be a $C^*$-Kasparov $(A,B)$-module with grading $\ee$ such that $F=F^*$ and $||F||\leq 1$, set $U_F:=\ee\sqrt{1-F^2}+F$\index{$U_F$}. Then the $C^*$-quasihomomorphisms 
\begin{align*}(P_{\ee}\alpha,U_FP_{\ee}^\perp \alpha U_F)  \text{ and } (P_{\ee} \alpha,W_FP^{\perp}_{\ee}\alpha W_F)\end{align*}
are equivalent.
\end{Prop}
\begin{proof}
Set 
$$U_F':=\scalebox{0.8}{$\mb F& \sqrt{1-F^2}\\ \sqrt{1-F^2}& -F\me$},\; W_F':=\scalebox{0.8}{$\mb F&(1-FF^*)\\ (1-F^*F)& (F^*F-2)F^*\me$}.$$
Then $(E\oplus E^{op},\phi\oplus 0, U_F')$ and $(E\oplus E^{op},\phi\oplus 0, W_F')$ are perturbations of the same module,  hence equivalent. A straightforward calculation shows that the quasihomomorphism induced by $(E\oplus E^{op},\phi\oplus 0, U_F')$ is equivalent to $(P_\ee\alpha,U_F P_{\ee}^\perp\alpha U_F)$; similarly for $W_F$ and $W_F'$. One may also see this by remarking that $U_F$ and $W_F$ can be obtained from passing to the Fredholm picture, i.e., regrouping $E\oplus E^{op}$ into even and odd parts by applying a permutation matrix. The conjugation of $U_F'$ and $W_F'$ by this permutation matrix yields operators with off diagonal entries $U_F$, $U_F^*$ and $W_F$, $W_F^*$, respectively. \note{ The module used in the definition is $E\oplus E^{op}=E^{(0)}\oplus E^{(1)}\oplus E^{(0)}\oplus E^{(1)}$. Hence we apply \scalebox{0.8}{$\mb 1&\\& \mb &&1\\&1&\\1&&\me\me$}. This interchanges the second and last row and second and last column. The result, on $U_F'$ is thus  of the form
\scalebox{0.8}{ $  \mb 	 &   \mb \sqrt{\cdot } &T*\\-T&\sqrt{\cdot }\me\\\mb \sqrt{\cdot } &T*\\-T&\sqrt{\cdot }\me& \me $  }}{}
\end{proof}
There is an obvious sum operation on   abstract Kasparov modules, given by 
$$(\alpha,\bar\alpha,U)+(\alpha',\bar\alpha',U'):=(\alpha\oplus\alpha',\bar\alpha\oplus\bar\alpha',U\oplus U').$$ Further every homomorphism $\phi\in\LC(\mcA,\mcB)$ yields a canonical abstract Kasparov module $(\mcB,\phi,0)$ and also an operation $\phi^*$ on abstract Kasparov modules in the obvious way. The same statements hold for Kasparov modules. In particular, if $\psi:\mcA'\to \mcA$ is a homomorphism and $(\hat\mcB,\phi,F)$ an $\LC$-Kasparov module, then we write $\psi^*(\hat\mcB,\phi,F):=(\hat\mcB,\phi\circ\psi,F)$.
\begin{Def} Two $\LC$-Kasparov modules $(\hat\mcB_0,\phi_0,F_0)$, $(\hat\mcB_1,\phi_1,F_1)$ between locally convex algebras $\mc{A}$ and $\mc{B}$ are called 
\begin{itemize}
\item $\mc{B}$-perturbations of one another if $\hat\mcB_0=\hat\mcB_1$, $\phi_0=\phi_1$ and $a\mapsto\phi(a)(F-F'),(F-F')\phi(a)$ is $\mcB$-valued and continuous
\item diffotopic, if there is an $\LC$-Kasparov $(\mcA,\scy\mcB)$-module $(\hat\mcB,\phi,F)$ such that $\scy\mcB$ is an ideal in $\hat\mcB$, and for $i=1,2$ 
$$(\pi_i)_*(\hat\mcB,\phi,F):=(\hat\mcB/\sCo_i\mcB,\pi_i\circ \phi,\pi_i( F))=(\hat\mcB_i,\phi_i,F_i)$$
where $\pi_i:\hat \mcB\to \hat \mcB/\sCo_i\mcB$ is the quotient map.
\end{itemize} 
\end{Def}
\begin{Prop}\label{perturbKM} $\mcB$-perturbation is a weaker equivalence relation than diffotopy, and the transformation $Qh$ from $\LC$-Kasparov modules to $\LC$-quasihomo-morphisms preserves diffotopy.
\end{Prop}
\begin{proof}
Let $(\hat\mcB,\phi,F)$, $(\hat\mcB,\phi,G)$ be $\mcB$-perturbations of one another. Let $g\in\sCo$ with $g(1)=1$, set $f:=1-g$, define $\tilde F:=fF+gG\in\scy \mcB$. Then multiplication by $\tilde F$ on $\scy\mcB$ is continuous because multiplication by $F$ and $G$ and elements of $\scy$ is. Further
\begin{align*}a(\tilde F^2-1)=&a(f^2(T^2-1)+g^2({T'}^2-1)+ fg\left(({T'}^2-1)+(T^2-1)\right)\\&+fg\left((T-T')T'+(T'-T)T\right)
\end{align*}
which is continuous in $a$.  The rest is obvious.

\end{proof}
The next example shows that the induction process is correctly normalized.
\begin{Ex} Let $(\hat\mcB,\phi,F)$ be an $\LC$-Kasparov $(\mC,\mcB)$-module with grading $\ee$, where $\mcB$ is some locally convex algebra, such that $\phi$ is unital and $\bar S S=1$, where $F=\mb &\bar S\\S&\me$ with respect to $P_\ee$. Then we get
$$W_F=\mb 0& \bar S\\S& S\bar S-1\me ,\;\;\; (P^\perp_{\ee}\phi )^{W_F}=\mb 1&0\\0&1-S\bar S\me.$$
Hence the map $(P^\perp_{\ee}\phi )^{W_F}-P_{\ee}\,\phi$ is a homomorphism orthogonal to $P_\ee\phi$, namely the map 
$$\mC\to\mcB,\; \lambda\mapsto \lambda  \mb 0& \\ &1-S\bar S\me.$$ 
Thus for every split exact functor $H$
$$H(\hat\mcB,\phi,F)=-H(1-S\bar S)$$
 by Proposition \ref{quasirules} -- \ref{orthosum} and 
\ref{negativequasi}.

If, on the other hand, $S\bar S=1$, we obtain  the homomorphism 
$$P_\ee\phi-(P_\ee^\perp\phi)^{W_F}:\lambda\mapsto \lambda (1-\bar S S)\oplus 0,$$ 
and by Proposition \ref{quasirules} \ref{orthosum}
$$H(\hat B,\phi,F)=H(1-\bar S S).$$ 
One could also deduce the second from the first by a rotation in matrices and using diffotopy invariance.
\end{Ex}
The following is an abstract variant of \cite{KaspOp}, Theorem 5:
\begin{Prop}\label{indeximports} Let $(\bo,\phi,F)$ be a $(\mC,\mc{L}^p)$-module, and $H$ a split exact, diffotopy invariant and $\mc{L}^p$-stable functor. Then $H(Qh(\bo,\phi,F)\otimes_\pi \mcA))=n\theta$ for every locally convex algebra $\mcA$, where $\theta:H(\mcA)\to H(\mcA\otimes_\pi\mc{L}^p)$ denotes the stabilisation map, and $n$ is the Fredholm index of $F$ viewed as an operator on $\phi(1)\hsp$.
\end{Prop}
\begin{proof} It suffices to prove the case with $\mcA=\mC$, the general case follows by applying the result to the functor $H(\,\cdot\,\potimes \mcA)$. Setting $P:=\phi(1)$ and replacing by the module 
$$(P\bo P,P\phi,PFP),$$ we may assume that $\phi$ is unital (because $H(P^\perp\bo P^\perp,P^\perp\phi,P^\perp FP^\perp)=0$). With respect to the grading, $F= \scalebox{.8}{$\mb &S\\T&\me$}$. By hypothesis, $T$ is Fredholm and hence has closed cokernel. 

Assume that $T$ has negative index. Without loss of generality, it is injectif, and we can define the bounded operator $T'$ as $T^{-1}$ on $\Im(T)$ and zero on $\Im(T)^{\perp}$. We have 
$$(T'-S)T=1-ST\in\mc{L}^p,$$
and as $T$ is invertible modulo $\mc{L}^p$, $T'-S\in\mc{L}^p$ follows. Therefore the module $(\bo, \phi, \scalebox{.8}{$\mb&T'\\T&\me$})$ is equivalent to the former one. Now the above example shows that 
$$H(\bo,\phi,F)=-H(1-TT')=-n\theta$$ 
by Proposition \ref{functorgeneralities} \ref{ntimesstab}.

In case that $T$ has negative index, we reduce to the former case by passing to the module $(\hat\mcB,\phi,-F)$ with grading $-\ee$.

\end{proof}

\subsection{Example: The Bott element}\label{smoothbottel}
We now construct the smooth analogue of the Bott elements. We assume in this section that the dimension $n$ is even. The point is that otherwise we would have to work with graded functors, as certain gradings would not be inner. 

We recall that the $C^*$-Kasparov $(\mC,S^n\otimes\mC_n)$-module $y_n$ was defined by Kasparov (\cite{KaspOp}, see also \cite{kkT}) as $(S^n\otimes\mC_n,1,q(D_2))$, where $D_2$ denotes the operator of multiplication by the inclusion $\mR^n\hookrightarrow \mC_n$ and $1$ denotes the action of $\mC$ given by scalar multiplication. In \cite{KaspOp}, the grading is the one coming from the natural grading of the Clifford algebras. We may assume that the algebra $S^n\otimes\mC_n$ is trivially graded and the grading on the Hilbert module $S^n\otimes\mC_n$ is given by left multiplication with the volume element. This does not change the class in $KK$-theory, because it corresponds to applying a graded Morita context. In fact, more generally:
\begin{center} If $A$ is a $C^*$-algebra with an inner grading $\kappa_U$ such that $U^2=1$, then there is a graded Morita context between $A$ and $A$ with the trivial grading.
\end{center} The context is given by letting $(A, \kappa_U)$ (with grading) act on the Hilbert $(A,\id_A)$-module (trivially graded) $A$ with grading induced by left multiplication with $U$. Taking Kasparov products with the corresponding $KK$-equivalence transforms the grading exactly to the ones used above. Note that this does \textit{not} signify forgetting the grading (in fact,  $x_n$ is trivial in $KK$ if we just view it as ungraded).

As we need to define an $\LC$-Kasparov $(\mC,\ssu^n\otimes\mC_n)$-module, the operator $q(D_2)$ used in the $C^*$-setting is inadequate, as $1-q(D_2)^2=(1+D_2^2)^{-1}$ is not a Schwartz function. The idea, in dimension $1$, is that it suffices to replace the inclusion $f:\mR\to \mC_1$ by the function $t\mapsto (0 ,te^{t^2})$, or any other odd real-valued function $\fun$ growing sufficiently fast at infinity\ifthenelse{\boolean{notes}}{: In the unbounded picture, all we have to do is find a function such that $1-q(f)^2=(1+f^2)^{-1}$ is Schwartz, as the commutator-condition is trivial}{};  in arbitrary dimension we define $\LC$-Kasparov $(\mC,\ssu^n\otimes \mC_n)$-modules \index{$y_n^\infty$}
$$y_n^\infty:=(\fn_b(\mR^n)\otimes\mC_n,1,q(\fun(D_2))).$$ 
We then have the following weak multiplicativity property   (in contrast to the classical setting)
\begin{Lem} Let $H$ be a split exact, diffotopy invariant $M_2$-stable functor, $m,n\in 2\mN$. Denote $D_{n,2}$ the operator $D_2$ in dimension $n$. Then 
\begin{align*} &H((\fn_b(\mR^{m+n})\otimes\mC_{m+n},1,q(h(D_{m,2})+h(D_{n,2}))))\\
=&H(y_{m+n}^\infty)\\
=&H(y_m^\infty\otimes (\mc{S}(\mR^n)\otimes\mC_n))\circ H(y_n^\infty).
\end{align*}
\end{Lem} 
\begin{proof}
The first equality  follows by factoring over $K$-theory or from a diffotopy using the operator
$$q(e^{t||y||^2}\,h(c_+(x))+e^{t||x||^2}\,h(c_+(y))),$$
where $x\in \mR^m$, $y\in\mR^n$.

To see the second equality, note that in $K$-theory the class of $y_{m+n}^\infty$  is the outer product of $y^\infty_m  $ and $y_n^\infty$. Further, $H(y_{m+n}^\infty)$ coincides with the action of the $K$-theory class $[y_{m+n}^\infty]$ by multiplication, as defined in Proposition \ref{pairing}. As the pairing is compatible with outer products (Proposition \ref{pairing}), we get the above statement.
\end{proof}
Note also that the operators $q(h(D_2))$ are equivariant with respect to the $O(n)$-action; this is not true for the operators built up from the one-dimensional ones, which do however have a representative given by an equivariant module according to the above Lemma.

\section{Spectral triples}

We come now to the notion of spectral triple, which is central in noncommutative geometry. Recall that $[T,D]$ is called bounded, where $D$ is a regular self adjoint operator on a Hilbert space, and $T$ a bounded operator, if $T$ preserves the domain of $D$ and $[T,D]$ extends to an element of $\bo$. \note{This is not the same condition as the one used in \cite{BlackK}, Definition 17.11.1, and it is shown in \cite{MR2679393} that the definition from \cite{BlackK} gives a contradiction to elementary properties of $KK$.}{}
\begin{Def}\label{Spectraltriple} A spectral triple is given by an involutive algebra $\mathfrak{A}$, a representation $\phi:\mathfrak{A}\to \bo$ on a Hilbert space $\hsp$ and a selfadjoint (unbounded) operator $D$ on $\hsp$ such that for all $a\in \mathfrak{A}$
$$[\phi(a),D]\in\bo,\phi(a)(1+D^2)^{-1/2}\in \co.$$
If $\hsp$ is graded and $\phi$ is even with respect to the induced grading on $\bo$, and $D$ is odd, then the spectral triple is called even. If $\hsp$ is ungraded, the spectral triple is called an odd spectral triple. 
\end{Def}
\begin{Def}\label{strongsummability}A spectral triple is called $p$-summable if $\phi(a)(1+D^2)^{-1/2}\in \Lp$, and finitely summable if it is $p$-summable for some $p\in\mN$. 
\end{Def}
Because $(1+D^2)^{-1}=R_i(D)R_{-i}(D)$, we see that it is equivalent to demand $a(1+D^2)^{-\nicefrac{1}{2}}\in\mc{L}^p$ or $aR_i(D)\in\mc{L}^p$. In \cite{arXiv:0810.2088v1}, the topology on the algebra $\mathfrak{A}$ is reconstructed from the axioms. We go the other way, defining a spectral triple as an "unbounded $\LC$-Kasparov $(\mcA,\mc{L}^p)$-module" represented on some Hilbert space: 
\begin{Def} A continuous $p$-summable spectral triple is a spectral triple over a locally convex algebra $\mcA$ such that $[D,\phi(a)]$ is continuous if viewed as a function of $a$ with values in the bounded operators, and $\phi(a)(1+D^2)^{-1/2},(1+D^2)^{-1/2}\phi(a)$ are continuous as  functions of $a$ with values in $\mc{L}^p$.
\end{Def}
We will need a consequence of the Baaj-Julg integral formula for the proof of Proposition \ref{transformspectral} below. We set \index{$q(x)$}$q(x):=\frac{x}{\sqrt{1+x^2}}$ for all $x\in\mR$. $q$ is called the Woroniwicz transform - just like the Cayley-transform, it associates a bounded operator to an unbounded one, hence it can be used to define a functional calculus for unbounded operators (\cite{MR1096123}). We have the Baaj-Julg formula from \cite{BaajJulg}
\begin{Lem} Let $D$ be a regular self adjoint unbounded operator on a Hilbert $B$-module $E$. Then as an integral in $\mathbb{B}(E)$ 
$$(1+D^2)^{-1/2}=\frac{1}{\pi}\int_0^\infty(1+t+D^2)^{-1}\frac{dt}{\sqrt{t}}.$$
Hence for every $\xi\in E$
$$D(1+D)^{-1/2}\xi=\frac{1}{\pi}\int_0^\infty D(1+t+D^2)^{-1}\xi\frac{dt}{\sqrt{t}}.$$
\end{Lem}

\begin{Lem}\label{strongintegral} Let $E$ be a Hilbert module, $D$ a regular self-adjoint operator, $a\in \mathbb{B}(E)$. If $[D,a](1+D^2)^{\frac{\ee-1}{2}}$ is bounded for some $\ee>0$, then
\begin{align*} [q(D),a]=\frac{1}{\pi}&(\int_0^\infty (1+t)(1+t+D^2)^{-1}[D,a](1+t+D^2)^{-1}\frac{dt}{\sqrt{t}}\\
+&\int_0^\infty D(1+t+D^2)^{-1}[a,D]D(1+t+D^2)^{-1}\frac{dt}{\sqrt{t}}),
\end{align*}
as a uniformly converging integral of bounded operators.
\end{Lem}
Here the hypothesis that $[D,a](1+D^2)^{\frac{\ee-1}{2}}$ is bounded means that $a$ preserves the domain of $D$ and that $[D,a](1+D^2)^{\frac{\ee-1}{2}}$ is an adjointable operator  on $\Dom((1+D^2)^{\ee/2})$.
\begin{proof}
Set $T_t:=(1+t+D^2)^{-1}$, let $a\in A$ be homogeneous. Then we have
\begin{align*}
[DT_t,a]=& DT_t a-(-1)^{\partial a\partial D}aDT_t \\
=& DT_ta T_t^{-1}T_t-(-1)^{\partial a\partial D}T_t^{-1}T_t aDT_t \\
=& (1+t)( DT_taT_t-(-1)^{\partial a\partial D}T_taDT_t)\\
+& DT_taD^2T_t-(-1)^{\partial a\partial D} D^2T_taDT_t)\\
=& (1+t) (T_t[D,a]T_t)+( DT_t[a,D]DT_t).
\end{align*}
Estimating the norms, we see that for every $t$
\begin{align*} ||T_t[D,a]T_t||\leq||T_t||\,||[D,a](1+D^2)^{\frac{\ee-1}{2}}||\,||(1+D^2)^{\frac{1-\ee}{2}}T_t||\leq C\,(1+t)^{-\frac{1+\ee}{2}},
\end{align*}
for a positive constant $C$, and similarly for the second term. Thus
$$[q(D),a]=\frac{1}{\pi}\int_0^\infty ((1+t)(T_t[D,a]T_t)+(DT_t[a,D]DT_t))\frac{dt}{\sqrt{t}}$$
as an integral with values in the bounded operators.
\end{proof}

Note that a linear function $\phi:\mcA\to\mc{L}^p$  is continuous iff $a\mapsto \phi(a)^*\phi(a)$ is continuous into $\mc{L}^{p/2}$ at $0$.
\note{For this, note that if we denote by $\lambda_i(T)$ the sequence of eigenvalues of a positive operator arranged in decreasing order, then by definition 
\begin{align*}\mu_n(A):=\lambda_n(|A|)=\sqrt{\lambda_n(A^*A)},\; ||A||_p:=(\sum_n{\mu_n(A)}^p)^{\nicefrac{1}{p}}\\
\Rightarrow ||A^*A||_{\frac{p}{2}}=(\sum_n \lambda_n (A^*A)^{\nicefrac{p}{2}})^{\nicefrac{2}{p}}=(\sum_n \lambda_n(|A|)^p)^{\nicefrac{2}{p}}=||A||_p^2.\end{align*}
Now assume $\phi^*\phi$ is continuous, $\phi$ linear, and $a_n$ a sequence in $A$ converging to zero. Then 
$$||\phi(a_n)||_p^2=||\phi^*(a_n)\phi(a_n)||_{\frac{p}{2}}\to 0.$$
In other words, if $p$ is a continuous seminorm on $A$, and $||\phi^*(a)\phi(a)||_{\frac{p}{2}}\leq C p(a)$ for all $a$, then
$||\phi(a)||_{p}=||\phi^*(a)\phi(a)||^{\frac{1}{2}}_{\frac{p}{2}}\leq \sqrt{C}\sqrt{p(a)}$, and $\sqrt{p}$ is a continuous seminorm on $A$ by continuity of the square root.
}{} 
Hence for a continuous spectral triple $a\mapsto (1+D^2)^{-\nicefrac{1}{4}}\phi(a)$ is also continuous (with values in $\mc{L}^{2p}$). We have:
\begin{Prop}\label{transformspectral} If $(\hsp,\phi,D)$ is a continuous $p$-summable spectral triple over a locally convex algebra $\mcA$ such that the multiplication $\mcA\times\mcA\to \mcA$ is surjective and open, then $(\bo,\phi,q(D))$ is an $\LC$-Kasparov $(\mcA,\mc{L}^{2p})$-module.
\end{Prop}
\begin{proof}
We set $T_t:=(1+t+D^2)^{-1}$. As $\mcA$ carries the identification topology for the product, it suffices to show that $(a,b)\mapsto [q(D),ab]$ is continuous on $\mcA\times\mcA$. Note first that for fixed $t\in\mR_{\geq 0}$
$$a\mapsto T_t^{\nicefrac{1}{4}}a= T_t^{\nicefrac{1}{4}}(1+D^2)^{\nicefrac{1}{4}}(1+D^2)^{-\nicefrac{1}{4}}a$$ is a product of multiplier of $\mc{L}^{2p}$ with a continuous function $\mcA\to \mc{L}^{2p}$, hence continuous. Now we use Lemma \ref{strongintegral} to develop 
\begin{align*}
[q(D),ab]=&\frac{1}{\pi}(\int (1+t) T_ta[D,b]T_t \frac{dt}{\sqrt{t}}+\int (1+t) T_t[D,a]bT_t\frac{dt}{\sqrt{t}}\\
+&\int DT_ta[D,b]DT_t\frac{dt}{\sqrt{t}}+\int DT_t[D,a]bDT_t\frac{dt}{\sqrt{t}}).
\end{align*}
These are norm-convergent integrals in $\mc{L}^{2p}$; for example:
\begin{align*} 
&    ||DT_ta [D,b]DT_t||_{\mc{L}^{2p}}\\
\leq & ||DT_t^{1/2}||\,||T_t^{1/4}||\,||T_t^{1/4}a||_{\mc{L}^{2p}}\,||[D,b]||\,||DT_t^{1/2}||\,||T_t^{1/2}||\\
\leq & \frac{C}{(1+t)^{3/4}} 
\end{align*}
for a constant $C$. Furthermore, the integrands are continuous functions of $a$ and $b$, hence the result follows from dominated convergence.

\end{proof}

\subsection{Example: The Dirac element}\label{Diracel}
We assume again that the dimension $n$ is even in order that the grading on the Clifford algebras be inner.

Recall again from \cite{KaspOp} or \cite{kkT} that $x_n$ denotes the $C^*$-Kasparov $(S^n\otimes\mC_n,\mC)$-module defined  by letting $S^n\otimes\mC_n$ act on $L^2$-forms on $\mR^n$ and with operator the Dirac operator. The grading on $S^n\otimes\mC_n$  is again inner; and we may therefore  pass again to to a $C^*$-Kasparov $(S^n\otimes \mC_n,\mC)$-module $x_n'$, where $S^n\otimes\mC_n$ is viewed as \underline{trivially} graded. Denote $\iota:\ssu^n\otimes\mC_n\to S^n\otimes \mC_n$ the inclusion.
\begin{Prop}\label{unitalizetoprod} The restriction $x_n^\infty:=\iota^*(x_n')$\index{$x_n^\infty$}, to $\ssu^n\otimes\mC_n$, of  $x_n'$  is an $\LC$-Kasparov $(\ssu^n\otimes\mC_n,\mc{L}^{2(n+1)})$-module.
\end{Prop} 
\begin{proof} This follows from Proposition \ref{transformspectral} once we have checked the hypotheses. That the commutators are bounded holds even for all smooth functions with bounded derivatives, and continuity of $f\mapsto[f,D]$ is obvious. It suffices to show $f(1+D^2)^{-1/2}\in \mc{L}^{(n,+)}$ for scalar $f$; for every function of compact support, the operator $f (1+D^2)^{-1}f^*$ is an element of $\mathcal{L}^{(n+1)/2}$, being a pseudodifferential operator of order $-1$ on a compact manifold (see for example \cite{Connesactionfunctional}, Theorem 1), hence $f(1+D^2)^{-1/2}\in \mathcal{L}^{n+1}$. In order to extend this to $\mc{S}$, let $h$ be a smooth function of compact support $K$ that is constant $1$ on the unit-cube; for $\alpha\in\mZ^n$, we denote by $t_\alpha$ the operator of translation by $\alpha$ - an isometry when restricted to $L^2$ - and set $K_\alpha:=t_\alpha(K)$. Define $H:=\sum_\alpha t_\alpha h$, and $g:=h/H$; then $\sum_\alpha t_\alpha(g)=1$. As the Dirac operator commutes with translations, for every smooth scalar $f$ and $\alpha\in\mZ^n$:
\begin{align*}
&||(1+D^2)^{-1/2} t_\alpha(g) f||_{n,+}\leq|| (1+D^2)^{-1/2} \circ t_\alpha\circ g\circ t_\alpha||_{n,+}\,||P_\alpha f||_{op}\\
=&|| (1+D^2)^{-1/2} g||_{n,+}\,\esssup(P_\alpha f)=|| (1+D^2)^{-1/2} g||_{n,+}\,||f\big|_{K_\alpha}||_\infty,
\end{align*}
where $P_\alpha=\chi(K_\alpha)$ is the projection determined by the support of $t_\alpha(g)$.
Now if $f$ is rapidly decreasing, then it is easily seen that 
$$\sum_\alpha||f\big|_{K_\alpha}||_\infty\leq C ||(1+x)^k f||_\infty$$
for an appropriate $k$ and a constant $C$, \note{ indeed: $$\sum_{\alpha}||f\big|_{K_\alpha}||_\infty\leq \sum_{\alpha}||(1+x)^{-k}\big|_{K_\alpha}||_{\infty}\,||((1+x)^k f)\big|_{K_\alpha}||_\infty\leq ||(1+x)^k f||_\infty\sum_\alpha||1+\alpha||^{-k}$$
where the last inequality \textit{should} follow by calculating the minimum of $K_\alpha$,}  and therefore 
$$||(1+D^2)^{-1}f||_{n,+}\leq C||(1+D^2)^{-1/2}g||_{n,+}||(1+x)^k f||_\infty,$$
which shows continuity of $f\mapsto f(1+D^2)^{-1/2}$.

\end{proof}
Alternatively, one may use the integral kernels of the operators and the following fact of independent interest from \cite{Simontrace}: If $K(x,y)=f(x)g(x-y)$, then $||T_K||_p\leq||f||_p||g||_p$, where $T_K$ denotes the operator with kernel $K$.

\section{Products of Kasparov modules and induced morphisms}
For brevity, we will call two quasihomomorphisms  $(\alpha,\bar\alpha)$ and $(\beta,\bar\beta)$  composable, if the target of the first is the domain of the second. We say they have a product for a given class of split exact functors if there is a quasihomomorphism $(\gamma,\bar\gamma)$  such that
$$H(\beta,\bar\beta)\circ H(\alpha,\bar\alpha)=H(\gamma,\bar\gamma)$$
for every functor in the class.

We use the analogous wordings for abstract Kasparov modules and $\LC$-Kasparov modules. 
Recall that for a given quasihomomorphism $(\alpha,\bar\alpha):\mc{A}\rrarrow \hat{\mc{B}}\trianglerighteq \mc{B}$, the algebra $\mc{D}_\alpha$ is defined as the topological 
vector space $\mc{A}\oplus \mc{B}$ with a twisted multiplication, and that it comes equipped with a canonical inclusion $\iota$ of $\mc{B}$ into $\mc{D}_\alpha$.
\begin{Prop}\label{productandextension} Let $(\hat\mcB,\phi_1,F_1)$ be a Kasparov $(\mathcal{A},\mc{B})$-module, $(\hat\mcC,\phi_2,F_2)$ a Kasparov $(\mc{B},\mc{C})$-module, $(\alpha,\bar\alpha):=Qh(\hat\mcB,\phi_1,F_1)$, $\mc{D}_\alpha$ and $\iota$ the to $\alpha$ associated algebra and inclusion. Let $H$ be a split exact functor. If there is a Kasparov $(D_\alpha,C)$-module $(\hat \mcC',\phi_2',F_2')$ such that $H(\hat\mcC,\phi_2,F_2)=H(\hat\mcC',\phi_2',F_2')\circ H(\iota)$, then 
$$H(\hat\mcC,\phi_2,F_2)\circ H(\hat\mcB,\phi_1,F_1)=H(\alpha^*(\hat \mcC,\phi_2,F_2))-H(\bar\alpha^*(\hat \mcC,\phi_2,F_2)).$$
In particular, $(\hat\mcB,\phi_1,F_1)$ and $(\hat\mcC,\phi_2,F_2)$ then have a product with respect to the class of split exact, $M_2$-stable functors.
\end{Prop}
\begin{proof} This follows from 
$$H(\iota)\circ H(\hat \mcB,\phi_1,F_1)=\alpha_*-\bar\alpha_*.$$
If $H$ is $M_2$-stable, we may use an obvious diagonal sum to present the latter.\end{proof}

\begin{Cor} Every two $\LC$-quasihomomorphisms $(\alpha,\bar\alpha):\mC\rrarrow \hat\mcB\trianglerighteq \mcB$ and $(\beta,\bar\beta):\mcB\rrarrow\hat\mcC\trianglerighteq\mcC$ 
have (up to $M_2$-stabilisation) a product with respect to split exact $M_2$-stable functors.
\end{Cor}
\begin{proof} By Lemma \ref{inductionseq} we can choose $(\alpha',\bar\alpha'):\mC\rrarrow M_2(\mcB^+)\trianglerighteq M_2(\mcB)$ such that $ H(\theta_B) \circ H(\alpha,\bar\alpha)=H(\alpha',\bar\alpha')$, where $\theta_B$ denotes the stabilisation by $M_2$-matrices. Denoting by $M_2(\phi)$ the inflation to matrices of a homomorphism $\phi$, we get from Proposition \ref{explicitstab}:
$$H(\beta,\bar\beta)\circ H(\alpha,\bar\alpha)=H(\theta_C)^{-1}\circ H(M_2(\beta),M_2(\bar\beta))\circ H(\alpha',\bar\alpha').$$
Now $H(M_2(\beta^+),M_2(\bar\beta^+))$ extends $H(M_2(\beta),M_2(\bar\beta))$.

\end{proof}
\section{Some theorems for split exact stable functors}
\subsection{Bott periodicity}
Let $n\in 2\mN$. Recall that $x_n^\infty$ and $y_n^\infty$ are defined in sections \ref{Diracel} and \ref{smoothbottel}. 
\begin{Th}\label{Bott periodicity} Let $H$ be a split exact, $\mc{L}^p$-stable functor. Then $H(x_n^\infty)$ and $H(y_n^\infty)$ are inverse to one another. 

More generally, for any locally convex algebra $\mcA$:
$$H(\mcA)\approx H((\mcS(\mR^n)\otimes\mC_n)\potimes \mcA).$$
\end{Th}
\begin{proof} By Theorem \ref{diffotopyinvariance}, $H$ is diffotopy invariant.
We first compute $H(x_n^\infty)\circ H(y_n^\infty)$; denote by $(\alpha,\bar\alpha)$ the $\LC$-quasihomomorphism associated to $x_n^\infty$ and by $(\beta,\bar\beta)$ the $\LC$-quasihomomorphism associated to $y_n^\infty$. Then $(\beta,\bar\beta)$ and $(\alpha,\bar\alpha)$ have a product by Corollary \ref{unitalizetoprod}.

By Proposition \ref{indeximports}, it suffices to calculate the index of the corresponding operator. Now the product-quasihomomorphism defines a class in $KK$-theory. It has index one by Proposition \ref{makeinvertible}, Remark \ref{differentdalpha}, the description of products in \cite{Note} and  Bott periodicity in $KK$. Thus
\begin{equation}\label{oneside}H(\alpha,\bar\alpha)\circ H(\beta,\bar\beta)=1.\end{equation}

We now show that $H(x_n^\infty)$ also has a left inverse.  
We proceed to calculate the product with the morphism induced by $$\mc{L}^p\otimes (\beta,\bar\beta):\mc{L}^p\rrarrow \mc{L}^p\otimes_\pi D_\beta\trianglerighteq \mc{L}^p\otimes_\pi\mcS^n\otimes\mC_n.$$
as defined in Definition \ref{outerprods}. 

To unclutter notation, set $\cliffs:=\mcS^n\otimes\mC_n$, $p:=2(n+1)$. Note that $\sigma_{{\cliffs},{\cliffs}}$ (the shift $\mu\otimes\nu\mapsto \nu\otimes\mu)$ is diffotopic to $\id_{{\cliffs}}\otimes \theta$, where $\theta$ is the tensor product of $f(x)\mapsto f(-x)$ with the automorphism of $\mC_n$ determined by $-\id_{\mR}$. We get, using  Proposition \ref{doublesplitouterproduct}
\begin{align*}
\rho:=H((\beta,\bar\beta)\otimes \mc{L}^p)\circ H(x_n^\infty)=&H({\cliffs}\otimes(\alpha,\bar\alpha))\circ H((\beta,\bar\beta)\otimes {\cliffs}).\end{align*}
Switching the factors in the tensor product (see again Proposition \ref{doublesplitouterproduct}), we have
\begin{align*}
\rho=H(\sigma_{\mc{L}^p,{\cliffs}})\circ H((\alpha,\bar\alpha)\otimes {\cliffs})\circ H(\sigma_{{\cliffs},{\cliffs}})\circ H((\beta,\bar\beta)\otimes {\cliffs}).
\end{align*}
Now we apply the diffotopy to replace $H(\sigma_{\cliffs,\cliffs})$:
\begin{align*}
\rho= H(\sigma_{\mc{L}^p,{\cliffs}})\circ H((\alpha,\bar\alpha)\otimes{\cliffs})\circ H(\cliffs\otimes\theta)\circ H((\beta,\bar\beta)\otimes{\cliffs}).\end{align*}
Another application of Proposition \ref{doublesplitouterproduct} yields
\begin{align*}
\rho=H(\sigma_{\mc{L}^p,{\cliffs}})\circ H(\mc{L}^p\otimes\theta)\circ H((\alpha,\bar\alpha)\otimes{\cliffs})\circ  H((\beta,\bar\beta)\otimes{\cliffs}).
\end{align*}
\note{Recall first: Proposition \ref{doublesplitouterproduct}: $H((\alpha_2,\bar\alpha_2)\otimes \mcB_1)\circ H(\mcA_2\otimes(\alpha_1,\bar\alpha_1))=H(\mcB_2\otimes(\alpha_1,\bar\alpha_1))\circ H((\alpha_2,\bar\alpha_2)\otimes \mcA_1)$. We get the equalities
\begin{enumerate}
\item Proposition \ref{doublesplitouterproduct} with $\mcB_1:=\mc{L}^p$ and $\mcA_2:=\mC$, $\mcB_2:=\cliffs$, $\mcA_1:=\cliffs$
\item is just shuffling everywhere (again \ref{doublesplitouterproduct})
\item is the diffotopy 
\item is again \ref{doublesplitouterproduct}-the morphism variant 
$$(\alpha,\bar\alpha\otimes \mcB_1)_*\circ (\mcA_2\otimes\phi)_*=(\mcB_2\otimes\phi)_*\circ((\alpha,\bar\alpha)\otimes \mcA_1)$$
with $\phi=\theta:\mcA_1=\cliffs\to \mcB_1=\cliffs$, $\mcA_2=\cliffs$, $\mcB_2=\mc{L}^p$.
\end{enumerate}}{}As we have shown above that $H(\alpha,\bar\alpha)\circ H(\beta,\bar\beta)=1$ for any functor, we may apply this to $H^{\cliffs}:=H(\,\cdot\,\potimes\cliffs)$ above to deduce that $\rho$ is invertible. Because $H(x_n^\infty)$ has right inverse $H(y_n^\infty)$ by \ref{oneside}, and is left invertible, because $\rho$ is invertible, $H(x_n^\infty)^{-1}=H(y_n^\infty)$.

If $\mcA$ is any locally convex algebra, we may apply the result to $H(\,\cdot\,\potimes \mcA)$ to obtain the result in general.
\end{proof}
\begin{Rem} If one does not want to refer to Bott periodicity in $KK$, one can proceed as follows:

Denote  $D_\beta\subseteq \mc{S}\hat\otimes\mC_n\oplus\mC$ the algebra associated to $(\beta,\bar\beta)$ (Lemma \ref{newalgebra}); then $D_\beta$ acts in a natural way on $\mc{H}$ for it may be viewed as a subalgebra of $\fn_b^\infty(\mR^n)\otimes \mC_n$.

Let $\bar F_1:=q(\bar D_1)\in\mathbb{B}(\mc{H})$, where $\bar D_1$ is the closure of $d+d^*+c_+(x)$. Then $(\hsp,c_+,\bar D_1)$ is a continuous spectral triple, and $\bar D_1$ has summable resolvent because we have the eigenbasis of Hermite polynomials (see also \cite{kkT}). Furthermore, $(D_1-\bar D_1)\phi(a)$ is bounded for every $a\in \mc{S}\otimes\mC_n$, therefore $q(D_1)$ and $\bar F_1$ are summable perturbations of one another. Consequently they induce the same morphisms under $H$ (Proposition \ref{perturbKM}), and we may instead calculate the product using the operator $\bar F_1$ (Lemma \ref{inductionvsdiffo}). Set $(\tilde\alpha,\bar{\tilde{\alpha}}):=Qh(\hsp,c_+,\bar F_1)$.

As $(\tilde\alpha,\tilde{\bar\alpha})$ extends to $D_\beta$, we see that  $(\beta,\bar\beta)$ and $(\tilde\alpha,\tilde{\bar\alpha})$ have a product with respect to $H$ (Proposition \ref{productandextension}). By Proposition \ref{indeximports}, it suffices to calculate the index of the corresponding operator. Now one may calculate, using the basis of Hermite polynomials, the product directly.

To calculate the product the other way around, one performs the rotation argument as in the proof above. Even though it is not clear that $\rho$ in the above proof has a product, the morphism
$$H(\sigma_{\mc{L}^p,{\cliffs}})\circ H(\mc{L}^p\otimes\theta)\circ H((\alpha,\bar\alpha)\otimes{\cliffs})\circ  H((\beta,\bar\beta)\otimes{\cliffs})$$ 
obtained at the end then has a product by the first half of the proof, and we are done.
\end{Rem}
\begin{Rem} After the first part of the proof, i.e., when we have shown 
$(x_n^\infty)_* \circ (y_n^\infty)_*=\theta_*$, where $\theta$ is the canonical map $\mC\to \mc{L}^{2(n+1)}$, then we may actually set 
$$z:=(y_n^\infty)_*\circ \theta_*^{-1}\circ (x_n^\infty)_*$$
and deduce that $z$ is an idempotent morphism from $H(\mc{S})$ to itself.

 If we are willing to use $kk^{\mc{L}^p}$, then it is therefore possible to use  a  "shortcut" and avoid the rotation trick as follows. We know that the coefficients $kk^{\mc{L}^p}(\mC,\mC)$ are $\mZ$ (see \cite{MR2207702}) and that $kk^{\mc{L}^p}$ satisfies Bott periodicity and is $M_2$ stable, and hence $kk^{\mc{L}^p}(\mc{S}(\mR^n),\mc{S}(\mR^n))=\mZ$. Consequently
$$z:=kk(\mC,y_n^\infty) \circ \theta^{-1} \circ kk(\mC,x_n^\infty)(1),$$
where $1$ denotes the unit in ${kk(\mc{S}(\mR^n),\mc{S}(\mR^n))}$, is an idempotent in $\mZ$. If it was zero, it would be zero on all of $\mZ=kk(\mc{S}(\mR^n),\mc{S}(\mR^n))$, and we would get a contradiction by
$$1=(\theta^{-1} \circ kk(\mC,x_n^\infty)\circ kk(\mC,y_n^\infty))^2(1)=0.$$

However, this argument does not carry over to the Thom isomorphism, and already fails for the theory $kk^{alg}$  (of \cite{MR1456322}) stabilized only by the smooth compact operators (one may adopt the above arguments to $kk^{alg}$).
\end{Rem}

\subsection{$\fn^\infty(X)$-linearity and a smooth Thom isomorphism} 

\begin{Def} If $\mcA$ is a locally convex algebra and $X$ a closed manifold, then we call $\mcA$ a $\fn^\infty(X)$-algebra if there is a homomorphism $\mu_A:\fn^\infty(X)\potimes\mcA\to \mcA$ that endows $\mcA$ with the structure of a left $\fn^\infty(X)$-module.
\end{Def}
\begin{Def} Let $\mcA$ and $\mcB$ be locally convex $\fn^\infty(X)$-algebras, $\hat \mcB$ an algebra that carries a left $\fn^\infty(X)$-module structure. An $\LC$-quasihomomorphism $(\alpha,\bar\alpha):\mcA\rrarrow\hat \mcB\trianglerighteq \mcB$ is called $\fn^\infty(X)$-linear, if $\alpha$ and $\bar\alpha$ are $\fn^\infty(X)$-linear.
\end{Def}
With this definition, we have the following
\begin{Lem}\label{xlinearity} Let $(\alpha,\bar\alpha):\mcA\rrarrow\hat \mcB\trianglerighteq \mcB$ be a $\fn^\infty(X)$-linear $\LC$-quasihomomor-phism. Then for every split exact functor
\begin{align*}
H(\alpha,\bar\alpha)\circ H(\mu_{\mcA})=H(\mu_{\mcB})\circ H( \fn^\infty(X)\otimes (\alpha,\bar\alpha)).
\end{align*}
\end{Lem}

We suppose throughout this section that $E$ is an oriented bundle of even dimension. 

For any bundle $E\twoheadrightarrow X$ we denote by $\mcE_E$ the canonical Hilbert $\fn_0(X)$-module  of longitudinal differential forms on $E$ that are locally continuous functions in the $X$-direction with values in the $L^2$-forms along the fibres. We denote by $h\in L^2(\mR^n)$  the function defined as $h(\xi):=\exp(-\frac{||\xi||^2}{2})$.  For every bundle, we get an inclusion 
$$\mcE_E\hookrightarrow \mcE_{E\oplus E^\perp},\; \omega\mapsto \omega\otimes h.$$
Thus we get an embedding
\begin{equation}\label{identify} \mathbb{K}_{\fn_0(X)}(\mcE_E)\subseteq \mathbb{K}_{\fn_0(X)}(\mcE_{E\oplus E^\perp})\approx \fn_0(X,\mathbb{K}(L^2(\mR^{p+q})\otimes\mC_{p+q})).
\end{equation}
We also denote by $D_E$\index{$D_E$} the longitudinal Dirac operator on $\mcE_E$.

In order to avoid the choice of an algebra of Schwartz sections, we assume for the rest of this paragraph that the base space is a \textit{compact} manifold. By $\mc{S}(E)$\index{$\mc{S}(E)$} we denote the smooth functions on $E$ that are Schwartz functions along the fibres, and by $\Gamma^\infty(E)$ the smooth sections of $E$. We denote $\cliffs_E$\index{$\cliffs_E$} the space of smooth sections of the longitudinal Clifford bundle of $E$ that are Schwartz in the direction of the fibres. We also denote $c_+:=\ee+\ee^*$ the fibrewise action of $\mc{S}(E)$ on $\mcE_E$.
\begin{Rem} If $E=P\times_{O(p)}\mR^p$ is a decomposition of $E$ as an associated bundle for a principal $O(p)$-bundle $P$ over $X$ obtained by choosing a Riemannian metric, then
 $$\cliffs_E=\Gamma^\infty (P\times_{O(p)}(\mcS(\mR^p)\otimes\mC_p)),$$
and $\mcS_E$ is thus the analogue of the function space used in \cite{kkT} and \cite{KaspOp}.
\end{Rem}
Using the identification from equation \ref{identify}, we get the following
\begin{Def} We define $x_E^\infty$\index{$x_E^\infty$} as the $\LC$-Kasparov $(\cliffs_E,\cliffs_X\potimes \mc{L}^p)$-module 
$$(\mathbb{B}_{\fn(X)}(\mcE_E),c_+,F_E:=q( D_E)),$$ 
and set $(\alpha,\bar\alpha):=Qh(x_E^\infty)$. We denote $y_E^\infty$ the $\LC$-Kasparov $(\cliffs_X,\cliffs_E)$-module obtained by taking $y_n^\infty$ fibrewise, and denote $(\beta,\bar\beta)$ the associated quasihomomorphism.
\end{Def}
We will use the following factorisation result in the proof of the Thom isomorphism:
\begin{Lem}\label{factorize} Let $E,E'\twoheadrightarrow X$ be two smooth bundles over a compact manifold $X$. If $\phi_{E,E'}:\mcS_{E\times E'}\to\mcS_{E\oplus E'}$ is the restriction of forms on $E\times E'$ to $E\oplus E'$, then there is a quasihomomorphism $(\alpha',\bar\alpha')$ such that
$$H(\mc{L}^p\otimes\phi_{X,E})\circ H((\alpha,\bar\alpha)\otimes \cliffs_{E'})=H(\alpha',\bar\alpha')\circ H(\phi_{E,E'}).$$
\end{Lem}
\begin{proof} This follows from Proposition \ref{quasirules}-\ref{quasiprecompose}, because $(\alpha,\bar\alpha)\otimes\cliffs_{E'}$ extends to a quasihomomorphism of $C^*$-algebras whose composition with the extension of $\mc{L}^p\otimes\phi_{X,E}$  factors over the extension of $\phi_{E,E'}$.
\end{proof}
\begin{Th}\label{Thomiso} Let $H$ be a split exact, $\mc{L}^p$ stable functor, $X$ a closed manifold and $E\twoheadrightarrow X$ an orientable real vector bundle of even rank n; let $\mcA$ be a supplementary locally convex algebra. Then there is an isomorphism
$$H(\mcA\potimes \cliffs_X)\approx H(\mcA\potimes\cliffs_E).$$
If $E$ is a $spin^c$-bundle, then 
$$H(\mcA\potimes\mcS(X))\approx H(\mcA\potimes\mcS(E)).$$
\end{Th} 
\begin{proof} Using the functor $H^{\mcA}:=H(\,\cdot\,\potimes\mcA)$, we may reduce the case for a general algebra $\mcA$ to the case $\mcA=\mC$. We also recall that $H$ is diffotopy invariant by Theorem \ref{diffotopyinvariance}.

Let $B\in K(\cliffs_E)$ denote the class obtained by restricting $y_E^\infty$ to $\mC$. Then  $B$ acts as multiplication as in $F_E^X(B):H(\cliffs_X)\to H(\cliffs_X\potimes \cliffs_E)$  by Proposition \ref{pairing}. Combining this with the multiplication $\cliffs_X\potimes\cliffs_E\to\cliffs_E$, we obtain a map
$$\mu(B):H(\cliffs_X)\to H(\cliffs_E).$$
We have
\begin{enumerate}
\item $H(\beta,\bar\beta)=\mu(B)$
\item $H(\alpha,\bar\alpha)\circ\mu(B)=H(K(\alpha,\bar\alpha)(B))$
\item $K(\alpha,\bar\alpha)(B)=1$\note{ This follows from
\[K(\alpha,\bar\alpha)(B)=K(\alpha,\bar\alpha)(\iota^*(y_E))=[\iota]\cap(y_E\cap[\alpha,\bar\alpha])=\iota^*(1_{KK(X,X)})=1_{K(X)}\]}{}
\end{enumerate}
Hence $H(\alpha,\bar\alpha)\circ H(\beta,\bar\beta)=1$ by Lemma \ref{xlinearity}:
\begin{align*}
H(\alpha,\bar\alpha)\circ H(\beta,\bar\beta)=& H(\alpha,\bar\alpha)\circ H(\mu_{\cliffs_E})\circ  F_E^X(B)\\
\stackrel{\ref{xlinearity}}{=} &H(\mu_{\cliffs_X})\circ H((\alpha,\bar\alpha)\otimes \cliffs_X)\circ F_E^X(B)\\
=& \id_{H(\cliffs_X)}\\
\end{align*}
where we apply Proposition \ref{pairing} and \ref{twoinduced} to see the last inequality.\note{ We here have neglected the stabilisation that has to be postcomposed with $H(\alpha,\bar\alpha)$. Still neglecting this,  we get more precisely
\begin{align*}
H(\alpha,\bar\alpha)\circ H(\beta,\bar\beta)=& H(\alpha,\bar\alpha)\circ H(\mu_{\cliffs_E})\circ  {}_HF_E^X(B)\\
\stackrel{\ref{xlinearity}}{=} &H(\mu_{\cliffs_X})\circ H((\alpha,\bar\alpha)\otimes \cliffs_X)\circ {}_HF_E^X(B)\\
=&H(\mu_{\cliffs_X})\circ H^X(\alpha,\bar\alpha)\circ {}_{H^X}F^{\mC}_E(B)\\
=&H(\mu_{\cliffs_X})\circ {}_{H^X}F_X^\mC(K(\alpha,\bar\alpha)(B))\\
=&H(\mu_{\cliffs_X})\circ {}_{H^X}F_X^\mC(1_{\cliffs_X})\\
=& \id_{H(\cliffs_X)}\\
=&1.
\end{align*}
This uses that $H(\alpha,\bar\alpha)\circ (B)=H(K(\alpha,\bar\alpha)(B))$ for any split exact functor $H$, which should follow from compatibility of the action of $K$-theory with morphisms.
}{}

Using again the rotation trick, but this time fibrewise, we get that $H(\alpha,\bar\alpha)$ is invertible. More in detail:
\begin{align*}
H(\mc{L}^p\otimes (\beta,\bar\beta))\circ H(\alpha,\bar\alpha)=&H(\mc{L}^p\otimes {\phi_{X,E}}_*(\cliffs_X \otimes B))\circ H(\alpha,\bar\alpha)\\
=&H(\mc{L}^p\otimes\phi_{X,E})\circ H((\mc{L}^p\potimes \cliffs_X)\otimes B)\circ H(\alpha,\bar\alpha)\\
\stackrel{\ref{outerprods}}{=}&H(\mc{L}^p\otimes \phi_{X,E})\circ H((\alpha,\bar\alpha)\otimes \cliffs_{E})\circ H(\cliffs_E\otimes B)\\
\stackrel{\ref{factorize}}{=}& H(\alpha',\bar\alpha')\circ H(\phi_{E,E})\circ H(\cliffs_E\otimes B)\\
=&H(\alpha',\bar\alpha')\circ H(\phi_{E,E}(\cliffs_E\otimes B))\\
=&H(\alpha',\bar\alpha')\circ H(\phi_{E,E}(B\otimes \cliffs_{E}))\\
\stackrel{\ref{factorize}}{=}&H(\mc{L}^p\otimes\phi_{X,E})\circ H^{\cliffs_{E}}((\alpha,\bar\alpha))\circ H^{\cliffs_{E}}(B  )\\
\end{align*}
\note{To see that $H(\phi_{E,E})\circ H(\cliffs_E\otimes B)=H(\phi_{E,E})\circ H(B\otimes\cliffs_E)$, there should be two possibilities. (To write
$H(\phi_{E,E})\circ H(\cliffs_E\otimes B)$ as a morphism coming from something}{}Thus again $H(\alpha,\bar\alpha)$ is left and right invertible, hence invertible with inverse $H(\beta,\bar\beta)$.

For the second part of the theorem, one applies the Morita context coming from the $spin^c$ structure.
\end{proof}
\begin{Rem} As in the case of Bott periodicity, one may also calculate the product more directly: 
Let $(\alpha,\bar\alpha):=Qh(x_n^\infty)$ and $(\beta,\bar\beta):=Qh(y_n^\infty)$. Then there are obvious fibred versions $(\gamma,\bar\gamma):=((\alpha,
\bar\alpha)\otimes \mC_n)_{P^\infty}$ and $(\delta,\bar\delta):=((\beta,\bar\beta)\otimes \mC_n)_{P^\infty}$ of these quasihomomorphisms obtained as in the $C^*$-case. As in the proof of smooth Bott 
periodicity, $(\delta,\bar\delta)$ may be replaced with $(\tilde \delta,\bar{\tilde\delta})$ which is the fibred version of the quasihomomorphism associated to the $\LC$-Kasparov module obtained as before by replacing the operator in the Dirac element by $\bar D$. Then $(\tilde \delta,\bar{\tilde\delta})$ extends to $(D_{\alpha\otimes\id_{\mC_n}})_{P^\infty}$. Hence $(\gamma,\bar\gamma)$ and $(
\tilde \delta,\bar{\tilde{\delta}})$ have a product. As the equivariant index of their product over a point is one, we get the result.

Using a fibred version of the rotation argument, we see that $(\tilde \delta,\bar{\tilde\delta})$ and $(\gamma,\bar\gamma)$ are inverse to each other.
\end{Rem}

\section{Applications}
We now sketch the possible applications of the results proved in this paper. Details will appear elsewhere.
\begin{enumerate}
\item Much effort has gone into developing an analogue of Kasparov's bivariant $K$-functor for more general algebras, see for example the papers by Cuntz concerning the functor $kk$, as well as those of Weidner (\cite{MR1014824}), or with a more moderate scope Phillips monovariant $K$-theory for Frechet algebras (\cite{MR1082838}). The Bott periodicity theorem   \ref{Bott periodicity} can be used to develop bivariant $K$-theories on categories more general than $C^*$-algebras. The recipe is similar to the techiques used by Higson in \cite{MR1068250} to construct $E$-theory. However, it seems necessary to approach the construction from a more "homotopic" viewpoint, and to avoid stabilisations. Essentially, one proceeds as follows: One forms the stable diffotopy category of the category of locally convex algebra, as put forward in \cite{MR658514}. One then inverts a certain class of morphisms, in order to obtain a split exact functor; stabilizing this functor by an appropriate operator ideal yields a stable homotopy invariant split exact functor. In a forthcoming paper, we show that, as a corollary of  Theorem \ref{Bott periodicity}, this is the universal split exact, homotopy invariant and stable functor.

It seems very unlikely at the moment that this functor coincides with Cuntz's $kk$, and thus that $kk$ is not the universal split exact functor, as one might expect if it was the analogue of Kasparov's $KK$. 
\item The Connes-Thom isomorphism for locally convex algebras: This  important result due to Connes (\cite{MR605351}) relates the $K$-theory of a crossed product of a $C^*$-algebras $A\rtimes_\alpha\mR^n$ by $\mR^n$ from the $K$-theory of $A$. The result was later proved elegantly by Fack and Skandalis in the bivariant setting (\cite{MR600000}). Using the techniques developed so far and the Bott periodicity theorem \ref{Bott periodicity}, it is possible to prove an analogous result for a smooth version of such crossed products. 

\item Further, there here are applications to index theory and pseudodifferential operators. The passsage from locally convex algebras to double split extensions used above may be used to associate elements in bivariant $K$-theory groups to pseudodifferential operators. Using the Thom isomorphism (Theorem \ref{Thomiso}), it is now possible to state an index theorem for any split exact functor with certain properties. In particular, this includes the program outlined recently by Cuntz which analyses index-theory in the setting of his functor $kk$. 
\end{enumerate}
\bibliographystyle{alpha}
\bibliography{../../Fullbib}
\end{document}